\documentclass[11pt]{amsart}
\usepackage[english]{babel}
\usepackage{amsmath,amsthm,amssymb}
\usepackage[mathscr]{eucal}
\usepackage{upgreek}
\usepackage{enumitem}
\usepackage[bookmarks,bookmarksnumbered]{hyperref}
\usepackage[initials]{amsrefs}
\usepackage[all]{xy}
\SelectTips{cm}{}

\numberwithin{equation}{section}

\newcommand{\bbN}{{\mathbb N}}
\newcommand{\bbQ}{{\mathbb Q}}
\newcommand{\bbE}{{\mathbb E}}
\newcommand{\bbR}{{\mathbb R}}
\newcommand{\bbZ}{{\mathbb Z}}

\newcommand{\mfg}{\mathfrak{g}}

\newcommand{\pr}{\operatorname{pr}}

\newcommand{\Ball}{\operatorname{B}}
\newcommand{\half}{\frac{1}{2}}

\newcommand{\SL}{\operatorname{SL}}

\newcommand{\overto}[1]{{\buildrel{#1}\over\longrightarrow}}

\newcommand{\acts}{\curvearrowright}

\newcommand{\len}{\operatorname{length}}
\newcommand{\ab}[1]{{#1}^{\operatorname{ab}}}
\newcommand{\scl}[2]{\operatorname{scl}_{#1}\left({#2}\right)}

\newcommand{\ssc}[2]{\frac{1}{#1}\bullet{#2}}

\theoremstyle{plain}
\newtheorem{mthm}{Theorem}

\newtheorem{mcor}[mthm]{Corollary}

\newtheorem{theorem}{Theorem}[section]
\newtheorem{lemma}[theorem]{Lemma}

\newtheorem{prop}[theorem]{Proposition}

\theoremstyle{definition}

\newtheorem{construction}[theorem]{Construction}
\newtheorem{example}[theorem]{Example}

\newtheorem{remark}[theorem]{Remark}

\numberwithin{equation}{section}

\begin{document}

\title[EFM on Nilpotent groups]{Asymptotic shapes for ergodic families of metrics\\ on Nilpotent groups}

\date{\today}

\author{Michael Cantrell}
\address{University of Illinois at Chicago, Chicago}
\email{mcantr2@uic.edu}

\author{Alex Furman}
\address{University of Illinois at Chicago, Chicago}
\email{furman@math.uic.edu}

\begin{abstract}
	Let $\Gamma$ be a finitely generated virtually nilpotent group.
	We consider three closely related problems: 
	(i) convergence to a deterministic asymptotic cone for an equivariant ergodic family of
	inner metrics on $\Gamma$, generalizing Pansu's theorem;
	(ii) the asymptotic shape theorem for First Passage Percolation for general (not necessarily independent)
	ergodic processes on edges of a Cayley graph of $\Gamma$;
	(iii) the sub-additive ergodic theorem over a general ergodic $\Gamma$-action.
	The limiting objects are given in terms of a Carnot-Carath\'eodory metric on 
	the graded nilpotent group associated to the Mal'cev completion of $\Gamma$.
\end{abstract}

\maketitle

\section{Introduction and statement of the main results} 
\label{sec:introduction_and_main_results}

Let $\Gamma$ be a finitely generated virtually nilpotent group.
The topic of this paper may be viewed from three slightly different perspectives:
\begin{itemize}
	\item[(i)] 
	As a generalization of the result of Pansu \cite{Pansu} showing that the asymptotic cone of an invariant inner metric $d$ on $\Gamma$
	is the Carnot group $G_\infty$  (the graded nilpotent Lie group associated with the Mal'cev completion $G$ of $\Gamma$) 
	equipped with a certain Carnot-Carath\'eodory metric $d_\infty$.
	Here we show that if one replaces a \emph{single invariant} metric $d$ by an \emph{equivariant ergodic family} $\{d_x \mid x\in X\}$ 
	of inner metrics on $\Gamma$, 
	then a.e. $(\Gamma,d_x,e)$ has the same asymptotic cone which is the Carnot group $G_\infty$ 
	equipped with a fixed Carnot-Carath\'eodory
	metric associated to certain averages of the family $\{d_x \mid x\in X\}$.
	\item[(ii)] 
	As a result about asymptotic shape for First Passage Percolation 
	model over $\Gamma$ driven by a general ergodic process $\Gamma\acts (X,m)$.
	(The case of independent times was recently studied by Benjamini and Tessera \cite{BT}).
	\item[(iii)]
	As a Subadditive Ergodic Theorem over a general ergodic probability measure preserving (hereafter p.m.p.) action $\Gamma\acts (X,m)$.
	Given a measurable function $c:\Gamma\times X\to \bbR$, satisfying 
	\[
		c(\gamma_1\gamma_2,x)\le c(\gamma_1,\gamma_2.x)+c(\gamma_2,x)\qquad (\gamma_1,\gamma_2\in\Gamma),
	\]
	and some additional conditions, we show that for a.e. $x\in X$ there is a unique limit to $c(\gamma,x)$ suitably normalized; 
	the limit is described on the Carnot group $G_\infty$ using a Carnot-Carath\'eodory construction.
\end{itemize}
Let us recall some facts about nilpotent groups.
Upon passing to a finite index subgroup and dividing by a finite normal subgroup, we assume hereafter that our group $\Gamma$
is a torsion-free nilpotent group with torsion-free abelianization $\ab{\Gamma}\cong\bbZ^d$; 
this adjustment does not change the problem - see \S \ref{sub:cocycle2cc} below. 
By the classical work of Mal'cev, a finitely generated, torsion-free, nilpotent group $\Gamma$
can be embedded as a discrete subgroup of a connected, simply connected, 
nilpotent real Lie group $G$ so that $G/\Gamma$ is compact.
Moreover, such an embedding $\Gamma<G$ is unique up to automorphisms of $G$.
This $G$ is often called the \textbf{Mal'cev completion} of $\Gamma$.
Associated with $G$ one has the \textbf{graded nilpotent} connected, simply connected, real Lie group $G_\infty$, 
that is constructed from the quotient spaces $\mfg^{i}/\mfg^{i+1}$
of the descending central series $\mfg=\mfg^{1}>\mfg^{2}>\dots>\mfg^{r+1}=\{0\}$ of the Lie algebra of $G$ (see below).
In particular, one can identify the abelianizations $\ab{G}:=G/[G,G]$ and $\ab{G}_\infty=G_\infty/[G_\infty,G_\infty]$ 
via $\mfg/\mfg^{2}\cong \mfg_\infty/\mfg_\infty^{2}$.
The graded Lie group $G_\infty$ admits a one parameter family $\{\delta_t \mid t>0\}$ of automorphisms that 
induce the linear homotheties $\times t$ on the real vector space $\ab{G}_\infty\cong \ab{\mfg}\cong\ab{\mfg}_\infty$.
Such a group $G_\infty$ (with the family of homotheties) is sometimes called a \textbf{Carnot group}.

\begin{example}
	The integral Heisenberg group $\mathbf{H}_\bbZ$ embeds in the $3$-dimensional real Heisenberg group 
	\[
		\mathbf{H}_\bbR=\left\{ M_{x,y,z}=
		\left(\begin{array}{ccc} 1 & x & z \\ 0 & 1 & y\\ 0 & 0 & 1\end{array}\right)\ \mid x,y,z\in\bbR\right\}
	\]
	by restricting $x,y,z$ to be integers. In this case $G=\mathbf{H}_\bbR$ is itself graded: $G=G_\infty$.
	The abelianization $\ab{G}$ is two dimensional, and $G\to \ab{G}$ is given by $M_{x,y,z}\mapsto (x,y)$.
	The similarities are given by 
	\[
		\delta_t(M_{x,y,z})=M_{tx,ty,t^2z}.
	\]
\end{example}

Let $d$ be a an inner right-invariant\footnote{One often considers left-invariant metrics; our choice of right-invariance is dictated by our notation for sub-additive cocycles.} metric $d$ on $\Gamma$, e.g. $d(\gamma_1,\gamma_2)=|\gamma_1\gamma_2^{-1}|_S$, 
where $|\gamma|_S$ is the length of a shortest word representing $\gamma$ using elements of a fixed generating set $S$ for $\Gamma$.
In \cite{Pansu} Pansu proved that associated with such $d$ there is a right-invariant proper metric $d_\infty$ on $G_\infty$, 
that is homogeneous in the sense that
\[
	d_\infty(\delta_t(g),\delta_t(g'))=t\cdot d_\infty(g,g')\qquad (g,g'\in G_\infty,\ t>0)
\]
and such that there is Gromov-Hausdorff convergence
\begin{equation}\label{e:GH}
 	(\Gamma,\frac{1}{t}\cdot d,e)\ \to\ (G_\infty,d_\infty,e).
\end{equation}
The metric $d_\infty$ is a result of Carnot-Carath\'eodry construction applied to a certain norm on $\ab{G}\cong \ab{G}_\infty$,
associated to $d$.

\medskip

To state our results we need to fix some further notations.
Let $\Gamma$ be a finitely generated, torsion-free, nilpotent group, denote by $G$ its Mal'cev completion, 
and by $G_\infty$ the associated Carnot group with homotheties $\{ \delta_t \mid t>0\}$.
Fix a right-invariant inner metric $d$ on $\Gamma$, e.g. a word metric as above, 
and let $d_\infty$ on $G_\infty$ be the associated Carnot-Carath\'eodory metric as in Pansu's theorem.

Given a function $f:\Gamma\to \bbR$ one can consider an asymptotic cone of its graph in $\Gamma\times \bbR$,
i.e. possible Gromov-Hausdorff limits of 
\[
	{\rm Graph}(f)=\{ (\gamma,f(\gamma)) \mid \gamma\in\Gamma\}\subset \Gamma\times \bbR
\] 
with $(e,0)$ being the marked point.
The functions $f$ that will appear below, will be special in several ways:
\begin{itemize}
	\item[(f1)] 
		the rescaled graphs ${\rm Graph}(f)$ actually have a unique Gromov-Hausdorff limit,
	\item[(f2)] 
		this limit is given by a graph ${\rm Graph}(\Phi)$
		of a function $\Phi:G_\infty\to \bbR$,
	\item[(f3)] 
		the function $\Phi:G_\infty\to \bbR$ appears in a Carnot-Carath\'eodory construction; in particular, it is homogeneous:
		$\Phi(\delta_t(g))=t\cdot \Phi(g)$ for $g\in G_\infty$ and $t>0$.
\end{itemize}
The convergence in (f2) implies that
\[
	t_i^{-1}\cdot f(\gamma_i)\to \Phi(g)\qquad \textrm{whenever}\qquad \scl{t_i}{\gamma_i}\to g \in G_\infty,
\]
where the latter relates to the Gromov-Hausdorff limit (\ref{e:GH}) with $t_i\to\infty$.
Let us say that two functions $f,f':\Gamma\to \bbR$ are \textbf{asymptotically equivalent} if 
\[
	f(\gamma)-f'(\gamma)=o(|\gamma|_S).
\]
Then $f$ satisfies (f1)-(f3) with $\Phi$ iff $f'$ does. 
One might say that $\Phi$ is \textbf{the unique homogeneous representative of the 
asymptotic equivalence class} of $f$ (here the uniqueness statement follows from the fact
that different homogeneous functions cannot be asymptotically equivalent).
 
\begin{mthm}\label{T:NilKingman}\hfill{}\\
	Let $\Gamma$ be a finitely generated virtually nilpotent group, $\Gamma\acts (X,m)$ an ergodic probability measure-preserving action, and
	$c:\Gamma \times X\to \bbR_+$ a measurable subadditive cocycle.
	Assume that 
	\begin{itemize}
		\item[{\rm (i)}] 
		For some $0<k\le K<+\infty$ one has $k\cdot |\gamma|_S\le c(\gamma,x)\le K\cdot |\gamma|_S$ for a.e. $x\in X$.
		\item[{\rm (ii)}] 
		For a.e. $x\in X$ for every $\epsilon>0$ there is a finite set $F\subset \Gamma$ so that for every $x'\in\Gamma.x$ 
		any $\gamma\in \Gamma$ can be written as $\gamma=\delta_n\cdots \delta_2\delta_1$ with $\delta_i\in F$ and 
		\[
			c(\delta_1,x')+c(\delta_2,\delta_1.x')+\dots+c(\delta_n,\delta_{n-1}\cdots\delta_1.x')\le (1+\epsilon)\cdot c(\gamma,x').
		\]
	\end{itemize}
	Then for a full measure set of $x\in X$ the functions $c(-,x):\Gamma\to \bbR$ are asymptotically equivalent to each other
	and are represented by a unique homogeneous function $\Phi:G_\infty\to\bbR$, that is obtained in the following construction.
\end{mthm}

\begin{construction}\label{cons:CC}
	Given a subadditive cocycle $c:\Gamma\times X\to \bbR_+$ over an ergodic action $\Gamma\acts (X,m)$ of a finitely generated virtually nilpotent group $\Gamma$.
	\begin{itemize}
		\item
		Up to finite index and finite kernel (once $\Gamma\acts (X,m)$ and $c:\Gamma\times X\to \bbR_+$ are adjusted accordingly) 
		we are reduced to the case that $\Gamma$ is a finitely generated nilpotent group that is torsion free and has torsion-free abelianization $\ab{\Gamma}$.
		\item 
		Define a subadditive function $\overline{c}:\Gamma\to \bbR_+$ by integration: 
		\[
			\overline{c}(\gamma):=\int_X c(\gamma,x)\,dm(x).
		\]
		\item 
		Define a subadditive function $f:\ab{\Gamma}\overto{}  \bbR_+$ by minimizing $F$ over fibers:
		\[
			f(\ab{\gamma}):=\inf\left\{ \overline{c}(\gamma_1) \mid \ab{\gamma}=\ab{\gamma}_1\right\}.
		\]
		\item 
		Define $\phi:\ab{\mfg}_\infty\to\bbR_+$ by viewing $\ab{\Gamma}$ as a lattice in the vector space 
		$\ab{\Gamma}\otimes\bbR$ 
		and observing that there is a unique homogeneous subadditive function (a possibly asymmetric norm)
		\[
			\phi:\ab{\Gamma}\otimes \bbR\overto{}\bbR_+
		\] 
		representing $f:\ab{\Gamma}\overto{} \bbR_+$.
		\item
		Define $\Phi:G_\infty\to\bbR_+$ to be the homogeneous function associated to $\phi$ viewed as 
		an asymmetric norm on $\ab{\Gamma}\otimes\bbR\cong\ab{G}\cong\ab{\mfg}\cong\ab{\mfg}_\infty$ and applying the Carnot-Carath\'eodory construction.
	\end{itemize}
	For more details see \S\S \ref{sub:CC-construction}--\ref{sub:cocycle2cc}.
\end{construction} 

Recall that a metric $d$ on a metric space $M$ is called \textbf{inner}  
if given $\epsilon>0$ there is $R<\infty$ so that for any $p,q\in M$ one can find $n\in\bbN$ and $p_0,\dots,p_n$ so that: 
$p_0=p$, $p_n=q$, $d(p_{i-1},p_i)<R$ for $1\le i\le n$, and
\[
	d(p_0,p_1)+d(p_1,p_2)+\dots+d(p_{n-1},p_n)\le (1+\epsilon)\cdot d(p,q).
\]
The following result can be viewed as a generalization of Pansu's result on a single right-invariant inner metric on $\Gamma$ 
to equivariant ergodic families of inner metrics.

\begin{mthm}\label{T:REM}\hfill{}\\
	Let $\Gamma$ be a finitely generated virtually nilpotent group, $\Gamma\acts (X,m)$ an ergodic p.m.p. action, and let $\{ d_x \mid x\in X\}$ be a measurable family of
	inner metrics on $\Gamma$ that is right-equivariant:
	\begin{equation}\label{e:equi-metrics}
		d_{x}(\gamma_1,\gamma_2)=d_{\gamma.x}(\gamma_1\gamma^{-1},\gamma_2\gamma^{-1})\qquad\qquad (\gamma,\gamma_1,\gamma_2\in\Gamma),
	\end{equation}
	and satisfies a uniform bi-Lipschitz estimate $0<a\le d_x/d\le b<\infty$ where $d$ is some right-invariant word metric on $\Gamma$.

	Then there exists a right-invariant homogeneous metric $d_\phi$ on $G_\infty$ so that for a.e. $x\in X$ there is
	Gromov-Hausdorff convergence
	\[
		(\Gamma,\frac{1}{t}\cdot d_x,e)\ \overto{GH}\ (G_\infty,d_\phi,e).
	\]
	Here $d_\phi(g_1,g_2)=\Phi(g_2g_1^{-1})$ with $\Phi$ from Construction~\ref{cons:CC} corresponding to 
	\begin{equation}\label{e:dx2c}
		c(\gamma,x):=d_x(e,\gamma).
	\end{equation}
\end{mthm}

One can also start from a sub-additive cocycle $c:\Gamma\times X\to \bbR_+$ and define 
\begin{equation}\label{e:c2dx}
	d_x(\gamma_1,\gamma_2):=c(\gamma_2\gamma_1^{-1},\gamma_1.x)\qquad (x\in X,\ \gamma_1,\gamma_2\in\Gamma).
\end{equation}
The resulting measurable family of functions is equivariant (as in (\ref{e:equi-metrics})), and each 
is a (possibly asymmetric) metric on $\Gamma$; 
condition~\ref{T:NilKingman}(ii) on $c$ corresponds to $d_x$ being inner.

\medskip
 
A natural example of an equivariant family of metrics as above appears in the following setting, 
known as \textbf{First Passage Percolation} model.
Fix a Cayley graph $(V,E)$ for $\Gamma$ defined by some finite symmetric generating set $S\subset\Gamma$
(so $V=\Gamma$ and $E=\{(\gamma,s\gamma) \mid \gamma\in\Gamma,\ s\in S\}$), 
and fix a $0<a<b<\infty$. 
Define $X:=[a,b]^E$ -- the space of functions $x:E\to [a,b]$; we think of $x_{(v,v')}$
as the time it takes to cross edge $(v,v')\in E$.
Since $\Gamma$ acts by automorphisms on $(V,E)$, it also acts continuously on the compact metric space $X$.
Let $m$ be some $\Gamma$-invariant ergodic Borel probability measure on $X$,
e.g. the Bernoulli measure $m=\mu^E$ where $\mu$ is some probability measure on $[a,b]$.
Every $x\in X$ defines the time it takes to cross any given edge $e\in E$ and we can define
\[
	d_x(v,v')=\inf \left\{ \sum_{i=1}^n x_{(v_{i-1},v_i)} \mid v_0=v, v_n=v', (v_{i-1},v_i)\in E \right\}
\]
to be the minimal travel time from $v$ to $v'$ in the particular realization $x\in X$ of the configuration of passage times of edges.
One is now interested in the \textbf{asymptotic shape} as $T\to\infty$ of the set 
\[
	\Ball^\Gamma_{x}(T):=\{ v\in V \mid d_x(e,v)<T \}
\]
of vertices that can be reached from the origin $e\in V$ in time $<T$, for a typical configuration $x\in X$.

\begin{mcor}\label{C:FPP}\hfill{}\\
	With the notations as above, there exists a homogeneous function $\Phi:G_\infty\to \bbR_+$, given in Construction~\ref{cons:CC}, 
	so that for $m$-a.e. $x\in X$ the sets $\Ball^\Gamma_{x}(T)$ are within $o(T)$-approximation from 
	\[
		\{ g\in G_\infty \mid \Phi(g)< T\}
	\]
	which is a $\delta_T$ image of a fixed set:
	\[
		\Ball^\Gamma_x(T)\ \sim\ \{ g\in G_\infty \mid \Phi(g)< T\}=\delta_T\left(\{ g\in G_\infty \mid \Phi(g)<1\}\right).
	\]
	Thus $\{ g\in G_\infty \mid \Phi(g)<1\}$ gives the asymptotic shape of a.e. $\Ball^\Gamma_x(T)$ rescaled by $T$ for $T\gg1$.
\end{mcor}
It follows from Theorem~\ref{T:NilKingman} that for $m$-a.e. $x\in X$ for any $\epsilon>0$ for $T>T(x,\epsilon)$
\[
	\{ g\in G_\infty\, |\, \Phi(g)<1-\epsilon\}\ \subset\ 
	\scl{T}{\Ball^\Gamma_x(T)}\ \subset\ \{g\in G_\infty\, |\,  \Phi(g)<1+\epsilon\}
\]
which is equivalent to the statement of the Corollary.

\medskip

Let us make some remarks about these results.

(1) I.~Banjamini and R.~Tessera \cite{BT} recently established the asymptotic shape theorem for the First Passage Percolation model 
	(Corollary~\ref{C:FPP}) for the case of an \emph{independent distribution on edges},
	i.e. the measure $m=\mu^E$. 
	In this framework their result is stronger: the assumption is weaker (rather than compact support the distribution $\mu$
	is assumed to have a finite exponential moment) and there is statement of a speed for the convergence to the asymptotic shape.
	However, the proofs, being based on probabilistic techniques, do not seem to apply to the general ergodic case
	as in Corollary~\ref{C:FPP}.

(2) The abelian case $\Gamma=\bbZ^d$ was considered by Boivin \cite{Boivin} in the context of First Passage Percolation
	as in Corollary~\ref{C:FPP},
	and then by Bj\"orkland \cite{Bjorklund} in the more general context of sub-additive cocycles as in 
	Theorem~\ref{T:NilKingman}. 
	Both results are proved under weaker integrability condition, namely
	$c(\gamma,-)\in L^{d,1}(X,m)$ (Lorentz space).
	This integrability condition is known to be sharp for sub-additive cocycles over 
	general ergodic $\bbZ^d$-actions \cite{BDD}.
	We note that in \cite{Bjorklund} no \emph{a priori} innerness assumption is imposed, but in retrospect
	it is satisfied.

(3)	Assumption (ii) in Theorem~\ref{T:NilKingman} (and the corresponding assumption of innerness of metrics in Theorem~\ref{T:REM}) 
	is necessary for the limit object $\Phi$ (and $d_\phi$) to be geodesic.
	Yet, it will become clear from the proof below that this condition is not needed for the inequality
	\[
		\limsup_{\scl{t}{\gamma}\to g}\ \frac{1}{t}\cdot c(\gamma,x)\le \Phi(g)
		\qquad\qquad (g\in G_\infty)
	\]
	for $m$-a.e. $x\in X$. In fact, the proof of this inequality (see \S \ref{sub:upperbound}) does not require the lower estimate in 
	Theorem~\ref{T:NilKingman}(i); it only uses the inequality $c(\gamma,x)\le K\cdot |\gamma|_S$, which is equivalent 
	to $c(\gamma,-)\in L^\infty(X,m)$ for $\gamma\in S$ a generating set for $\Gamma$.
	
(4)	It is possible that assumption (i) in Theorem~\ref{T:NilKingman} can be relaxed. 
	Yet, note that already in the Abelian case $\Gamma=\bbZ^d$ pointwise convergence requires $L^{d,1}(X,m)$-integrability.
		
(5)	Let $\Gamma<G$ and $G_\infty$ be as above. Theorems~\ref{T:NilKingman} and \ref{T:REM} 
	show that asymptotic shapes are classified by $\Phi$ (and $d_\phi$) for some unique, possibly asymmetric, norm 
	$\phi:\ab{\mfg}_\infty\to \bbR_+$. 
	The converse also holds: for every asymmetric norm $\phi$ the associated Carnot-Carath\'eodory $\Phi$ and $d_\phi$ 
	arise as an asymptotic shape for some cocycle over $\Gamma$, in fact from a subadditive function $F:\Gamma\to\bbR_+$.
	However, the question of which asymptotic shapes (equivalently norms) 
	can appear in First Passage Percolation with independent distribution on edges
	remains widely open.

\medskip

We would like to emphasize the following remark.
\begin{remark}
	An important example of subadditive cocycles over group actions are 
	\[
		c(\gamma,x)=\log\|A(\gamma,x)\|
	\]
	where $A:\Gamma\times X\to \SL_d(\bbR)$ is a matrix valued cocycle, 
	i.e. satisfies $A(\gamma_1\gamma_2,x)=A(\gamma_1,\gamma_2.x)A(\gamma_2,x)$.
	If $\Gamma$ is not Abelian, then the results of this paper do not apply to such cocycles -- 
	they systematically fail the innerness assumption.
	Yet, for any amenable group $\Gamma$ (in particular, nilpotent) one can describe the asymptotic 
	behavior of such cocycles: they are asymptotically equivalent 
	to a homogeneous subadditive function, namely the pull-back of a norm $\phi$ on the abelianization 
	$\ab{\Gamma}_1\otimes\bbR$ for some finite index subgroup $\Gamma_1<\Gamma$.
	More precisely, the norm has the form  
	\[
		\phi(\gamma)=\max_{1\le j\le d} |\chi_j(\ab{\gamma})|
	\]
	for some characters $\chi_1,\dots,\chi_d:\ab{\Gamma_1}\otimes\bbR\to \bbR$.
	In particular, if $\Gamma$ is a non-abelian nilpotent group, such homogeneous functions do not grow along the commutator subgroup unlike
	Carnot-Carath\'eodry metrics.
	This can be shown by applying a form of Zimmer's Cocycle Reduction lemma 
	(using the fact that $\Gamma$ is amenable) that allows one to bring the cocycle to an upper triangular 
	form and read off the growth from the diagonal.
\end{remark}

\medskip

\subsection*{Plan of the paper} \hfill{}\\
In Section~\ref{sec:the_asymptotic_cone} we recall some background on graded nilpotent Lie groups,
the Carnot-Carath\'eodory construction, Pansu's fundamental result on the asymptotic cone of nilpotent groups, 
and the construction of $\phi$, $\Phi$ and $d_\phi$ associated with the
sub-additive cocycle $c:\Gamma\times X\to\bbR_+$.
Section~\ref{sec:prep} contains two basic preliminary results needed for the proofs of the main theorems.
One results concerns approximation of admissible curves in the asymptotic cone $G_\infty$ by expressions of the form 
$T_k^n\cdots T_2^n T_1^n$ that we call \emph{polygonal paths} in $\Gamma$.
The second result (Theorem~\ref{T:erg-poly}) is of independent interest; 
it is an ergodic theorem for sub-additive cocycles along above mentioned polygonal paths.
With these preparations at hand we prove Theorems~\ref{T:NilKingman} and \ref{T:REM} in Section~\ref{sec:proof}. 

\subsection*{Acknowledgements}
The authors would like to thank Itai Benjamini, Romain Tessera, and Tim Austin for their interest in this work, useful comments and encouragement. This work was supported in part by the NSF grant DMS-1207803, Simons Foundation, and MSRI.



\section{The Carnot group as the asymptotic cone} 
\label{sec:the_asymptotic_cone}

In this section we recall Pansu's construction of the asymptotic cone $(G_{\infty},d_{\infty})$ of a finitely generated nilpotent group and give our construction of $(G_{\infty},d_{\phi})$, the almost sure asymptotic cone of the random (pseudo) metric space $(\Gamma,d_x)$.

\subsection{The graded Lie algebra/group} 
\label{sub:the_graded_lie_algebra_group}\hfill{}\\
Let $\Gamma$ be a finitely generated, torsion-free, nilpotent group and $G$ be its Mal'cev completion. 
In this subsection we construct the associated Carnot group. 
Since the Lie groups here are connected and simply connected, one can work with the Lie algebras.
Let $\mfg$ be the Lie algebra of $G$, and set
\[
	\mfg^1:=\mfg, \qquad \mfg^{i+1}:=[\mfg,\mfg^{i}].
\]
Being nilpotent, $G$ satisfies $\mfg^{r+1}=\{0\}$ for some $r \in \bbN$. 
Since $[\mfg^i,\mfg^j]\subset \mfg^{i+j}$ (and in particular $[\mfg^{i+1},\mfg^j],[\mfg^i,\mfg^{j+1}]\subset \mfg^{i+j+1}$)
the Lie bracket on $\mfg$ defines a bilinear map 
\[
	\left(\mfg^{i}/\mfg^{i+1}\right)\otimes \left(\mfg^{j}/\mfg^{j+1}\right)\ \overto{}\ (\mfg^{i+j}/\mfg^{i+j+1}),
\]
which can then be used to define the Lie bracket $[-,-]_\infty$ on 
\begin{equation}\label{e:frakv-decomp}
	\mfg_{\infty}:=\bigoplus_{i=1}^{r} \frak{v}_i,\qquad\textrm{where}\qquad \frak{v}_i:=\mfg^i/\mfg^{i+1}
\end{equation}
by extending the above maps linearly.
The resulting pair $(\mfg_\infty,[-,-]_\infty)$ is called the \textbf{graded Lie algebra} associated to $\mfg$.
Note that the linear maps
\[
	\delta_t:\mfg_\infty \to \mfg_\infty,\qquad \delta_t(v_1,\dots,v_{r})=(t\cdot v_1,t^2\cdot v_2,\dots,t^{r}\cdot v_{r}),
\]
satisfy $\delta_t([v,w]_\infty)=[\delta_t(v),\delta_t(w)]_\infty$ and $\delta_{ts}=\delta_t\circ \delta_s$ 
for $v,w\in\mfg_\infty$, $t,s>0$. 
Hence $\{\delta_t \mid t>0\}$ is a one-parameter family of automorphisms of the Lie algebra $\mfg_\infty$, 
and therefore define a one-parameter family of automorphisms of the Lie group $G_\infty:=\exp_\infty(\mfg_\infty)$,
that we will still denote by $\{\delta_t \mid t>0\}$. (Here we denote the exponential map $\mfg_\infty\to G_\infty$
by $\exp_\infty$ to distinguish it from $\exp:\mfg\to G$).

The graded Lie algebra naturally appears in the following limiting procedure.
Choose a splitting of $\mfg$ as a direct sum of vector subspaces 
\begin{equation}\label{e:V-decomp}
	\mfg=V_1\oplus\cdots\oplus V_{r},\qquad\textrm{so\ that}\qquad
	\mfg^i=V_i\oplus\cdots\oplus V_{r},
\end{equation} 
and choose a vector space identification
$L:\mfg\to\mfg_\infty$ so that $L(V_i)=\frak{v}_i$ the $i$th summand of $\mfg_\infty$.
For $t>0$ define the vector space automorphism $\sigma_t$ of $\mfg$ 
by setting $\sigma_t(v)=t^i\cdot v$ for $v\in V_i$ ($i=1,\dots,r$).
Then the Lie brackets $[-,-]_t$ on $\mfg$, given by 
\[
	[v,w]_t:=\sigma_{\frac{1}{t}}\left([\sigma_t(v),\sigma_t(w)]\right),
\]
defines a Lie algebra structure on $\mfg$ that is isomorphic to the original $[-,-]=[-,-]_1$ via $\sigma_t$.
However, one has
\[
	[L(v),L(w)]_\infty=\lim_{t\to\infty} [v,w]_t
\]
due to the fact that for $v\in V_i$, $w\in V_j$ the "leading term" of $[v,w]$ lies in $V_{i+j}$,
while the higher terms that belong to $V_{i+j+1}\oplus\cdots\oplus V_{r}$ 
become insignificant under the rescaling (see \cite{Pansu}).

Using the $\log:G\to \mfg$ and $\exp_\infty:\mfg_\infty\to G_\infty$ maps 
we obtain a family of maps 
\begin{equation}\label{e:sclt}
	\scl{t}{-}:\Gamma\ \overto{<}\ G\ \overto{\log}\ \mfg\ \overto{\sigma_{t^{-1}}}\ 
	\mfg\ \overto{L}\ \mfg_\infty\ \overto{\exp_\infty}\ G_\infty\qquad (t>0)
\end{equation}
that explains the asymptotic cone description of Pansu \cite{Pansu} as follows.
Let $d$ be an inner right-invariant metric $d$ on $\Gamma$ and  
\[
	(\Gamma,\frac{1}{t}\cdot d,e)\ \overto{GH}\ (G_\infty,d_\infty,e)
\]
the Gromov-Hausdorff convergence. Then a sequence $\gamma_i\in \Gamma$, rescaled by $t_i^{-1}$
with $t_i\to\infty$ as $i\to\infty$, 
converges to $g\in G_\infty$ iff $\scl{t_i}{\gamma_i}\to g$ in $G_\infty$.
We shall often write 
\[
	g=\lim_{i\to\infty} \ssc{t_i}{\gamma_i}\qquad\textrm{instead\ of}\qquad \scl{t_i}{\gamma_i}\to g.
\]
The metric part of the statement shows that for $t_i\to\infty$ and $\gamma_i,\gamma'_i\in \Gamma$
\begin{equation}\label{e:lengthscaling}
	g=\lim_{i\to\infty}\ \ssc{t_i}{\gamma_i},\quad g'=\lim_{i\to\infty}\ \ssc{t_i}{\gamma'_i}
	\qquad
	\Longrightarrow
	\qquad
	d_\infty(g,g')=\lim_{i\to\infty} \frac{1}{t_i}\cdot d(\gamma_i,\gamma'_i).
\end{equation}
The limiting distance $d_\infty$ on $G_\infty$ is \textbf{homogeneous} in the sense that
\[
	d_\infty(\delta_s(g),\delta_s(g'))=s\cdot d_\infty(g,g')\qquad (g,g'\in G_\infty,\ s>0).
\]	
This distance is right-invariant (this follows from Lemma~\ref{L:multiplication}).
The distance $d_\infty$ appears in the sub-Finsler Carnot-Carath\'eodory construction discussed below.
Meanwhile let us point out two Lemmas.

\begin{lemma}\label{L:flattening-of-powers}
	For $\gamma\in\Gamma$ one has
	\[
		\lim_{n\to\infty}\ssc{n}{\gamma^n}=\exp_\infty(L\circ\pi\circ\log(\gamma))=\exp_\infty(\pi_\infty\circ L\circ \log(\gamma)),
	\]
	where $\pi:\mfg\to V_1$ and $\pi_\infty:\mfg_\infty:\mfg_\infty\to \frak{v}_1$ are 
	the linear projection corresponding to (\ref{e:V-decomp}), (\ref{e:frakv-decomp}). 
\end{lemma}
\begin{proof}
	Denote by $\pi_k:\mfg\to V_k$ ($k=1,\dots,r$) the linear projections according to (\ref{e:V-decomp}), so $\pi=\pi_1$. Then
	\[
		\ssc{n}{\gamma^n}=\exp_\infty\left(\sum_{k=1}^r \frac{1}{n^k}\cdot L\circ\pi_k\circ\log(\gamma^n)\right)
		=\exp_\infty\left(L\circ\pi_1\circ\log(\gamma)+\sum_{k=2}^r \frac{1}{n^{k-1}}\cdot L\circ\pi_k\circ\log(\gamma)\right)
	\] 
	and, since $n^{-k+1}\cdot L\circ\pi_k\circ\log(\gamma)\to0$ for $2\le k\le r$, the statement is clear.
\end{proof}

\begin{lemma}\label{L:multiplication}\hfill{}\\
	Given sequences $t_i\to\infty$, $\gamma_i,\gamma'_i\in\Gamma$ with $\ssc{t_i}{\gamma_i}\to g$ and $\ssc{t_i}{\gamma'_i}\to g'$
	then $\ssc{t_i}{\gamma_i\gamma'_i}\to gg'$.
\end{lemma}
\begin{proof}
	This follows from the Baker-Campbell-Hausdorff formula (cf. \S 3.3 and the proof of Lemma~5.5 in \cite{Bre}).
\end{proof}

\bigskip

\subsection{Carnot-Carath\'eodory constructions} 
\label{sub:CC-construction}\hfill{}\\
We follow \cite[(17)-(20)]{Pansu}.
Denote by $\ab{\mfg}_\infty$ the abelianization of the graded Lie algebra $\mfg_\infty$. 
It is isomorphic to the abelianization $\ab{\mfg}$ of $\mfg$, and can also be identified with the direct summand $\frak{v}_1$ of $\mfg_\infty$:
\[
	\ab{\mfg}_\infty\cong \ab{\mfg}\cong\frak{v}_1\ <\ \bigoplus_{i=1}^{r} \,\frak{v}_i=\mfg_\infty.
\]
Vectors in $\frak{v}_1<\mfg_\infty$ are called \textbf{horizontal}.
A tangent vector $v\in T_g G_\infty$ at $g\in G_\infty$ is horizontal if its right-translate under $g^{-1}$ is in 
$\frak{v}_1<\mfg_\infty=T_eG_\infty$.
Hence the horizontal vectors form a sub-bundle of the tangent bundle $TG_\infty$; this is a totally non-integrable sub-bundle because 
$\mfg_\infty$ is generated as a Lie algebra by $\frak{v}_1$. 
Let us say that a continuous piecewise smooth curve $\xi:[a,b]\to G_\infty$ 
whose tangent vectors $\xi'(t)$ are horizontal for Lebesgue a.e. $t\in[a,b]$ are \textbf{admissible}.
Any two points $g_1,g_2\in G_\infty$ can be connected by an admissible curve -- this follows from total non-integrability of 
the sub-bundle of horizontal vectors by Chow's theorem.

Let $\phi:\ab{\mfg}_\infty\to \bbR_+$ be an \textbf{asymmetric norm} (or rather a not necessarily symmetric norm), 
that is assume $\phi$ satisfies for all $v,w\in \ab{\mfg}_\infty$, $t>0$, and some $0<a\le b<\infty$:
\begin{equation}\label{e:asym-norm}
	\begin{split}
		&\phi(v+w)\le \phi(v)+\phi(w),\\
		&\phi(t\cdot v)=t\cdot \phi(v),\\
		&a\cdot\|v\|\le \phi(v)\le b\cdot\|v\|
	\end{split}
\end{equation}
for some reference Euclidean norm $\|-\|$. Such an asymmetric norm $\phi$ can be used to measure horizontal vectors in $TG_\infty$ by right-translating them back to
$\frak{v}_1<\mfg_\infty=T_eG_\infty$.
Given a curve $\xi:[\alpha,\beta]\to G_\infty$ as above its $\phi$-\textbf{length} is defined to be
\begin{equation}\label{e:phi-len}
	\len_\phi(\xi):=\int_\alpha^\beta \phi(\xi'(t)\xi(t)^{-1})\,dt.
\end{equation}
We define the $\phi$-\textbf{distance} by
\[
	d_\phi(g_1,g_2):=\inf\left\{ \len_\phi(\xi) \left|\ \xi\ \textrm{is\ an\ admissible\ curve\ from\ }
	g_1\ \textrm{to}\ g_2 \right.\right\}.
\]
Starting from a fixed Euclidean norm $\|-\|$ on $\ab{\mfg}_\infty$, one obtains the \emph{sub-Riemannian} metric $d_{\|-\|}$ on $G_\infty$,
also known as a \emph{Carnot-Carath\'eodory} metric; it is right-invariant, homogeneous with respect to the homotheties $\{\delta_t \mid t>0\}$,
and defines the usual topology on $G_\infty$.

For a general asymmetric norm $\phi:\ab{\mfg}_\infty\to \bbR_+$ as in (\ref{e:asym-norm}) we obtain
\[
	d_\phi:G_\infty\times G_\infty\ \overto{}\ \bbR_+,
\] 
that is a right-invariant, homogeneous, asymmetric metric, bi-Lipschitz to a Carnot-Carath\'eodory metric:
\begin{equation}\label{e:dphi-properties}
	\begin{split}
		&d_\phi(g_1g,g_2g)=d_\phi(g_1,g_2),\\
		&d_\phi(\delta_t(g_1),\delta_t(g_2))=t\cdot d_\phi(g_1,g_2),\\
		&d_\phi(g_1,g_2)\le d_\phi(g_1,h)+d_\phi(h,g_2),\\
		&a\cdot d_{\|-\|}(g_1,g_2) \le d_\phi(g_1,g_2)\le b\cdot d_{\|-\|}(g_1,g_2).
	\end{split}
\end{equation}
Being right-invariant $d_\phi$ is completely determined by the function
\[
	\Phi:G_\infty\to\bbR,\qquad \Phi(g):=d_\phi(e,g),
	\qquad
	d_\phi(g_1,g_2)=\Phi(g_2g_1^{-1}).
\]
This function $\Phi$ is sub-additive, homogeneous, and bi-Lipschitz to a Carnot-Carath\'eodory norm
\begin{equation}\label{e:Phi-properties}
	\begin{split}
		&\Phi(\delta_t(g))=t\cdot \Phi(g),\\
		&\Phi(g_1g_2)\le \Phi(g_1)+\Phi(g_2),\\
		&a\cdot d_{\|-\|}(e,g) \le \Phi(g)\le b\cdot d_{\|-\|}(e,g).
	\end{split}
\end{equation}
If $\phi$ is actually a norm, i.e. $\phi(-v)=v$, then $\Phi$ and $d_\phi$ are also symmetric: $\Phi(g^{-1})=\Phi(g)$ and $d_\phi(g_1,g_2)=d_\phi(g_2,g_1)$.
In this case $d_\phi$ is a sub-Finsler Carnot-Carath\'eodory metric on $G_\infty$ defined by the norm $\phi$.
Hereafter we shall use the term \textbf{Carnot-Carath\'eodory metric} (or just a \textbf{CC-metric}) 
when referring to a possibly asymmetric $d_\phi$
associated to $\phi$ as in (\ref{e:asym-norm}).
Pansu's metric $d_\infty$ on $G_\infty$, referred to in the previous section, 
is the Carnot-Carath\'eodory metric associated to a certain norm on $\mfg_\infty$,
that itself is determined by the given inner right-invariant metric $d$ on $\Gamma$ \cite{Pansu}.
The proof does not really use the symmetry assumption, so it can be applied almost verbatim to 
asymmetric norms. 
The infimum in the definition of $d_\phi(g_1,g_2)$ is achieved by a (unique) curve, that will be called a 
$d_\phi$-\textbf{geodesic}. But we shall use this fact only in reference to $d_\infty$ (or the classical $d_{\|-\|}$).

The notion of $\phi$-length can be extended to curves $\xi:[0,1]\to G_\infty$ that are $d_\infty$-\textbf{rectifiable}, 
i.e. ones for which
\[
	\sup \left\{  \sum_{j=1}^n d_\infty\left(\xi(s_{j-1}),\xi(s_j)\right) \mid n\in\bbN,\ 0=s_0<s_1<\dots<s_n=1\right\} \ <\ +\infty.
\]
Pansu shows (\cite{Pansu}) that such a curve is absolutely continuous, a.e. differentiable on $[0,1]$, and that
its derivative is a.e. horizontal, so the integral (\ref{e:phi-len}) makes sense. 
The $\phi$-length of such curves can also be defined by
\[
	\len_\phi(\xi)=\sup\left\{ \sum_{j=1}^n d_\phi\left(\xi(s_{j-1}),\xi(s_j)\right) \mid n\in\bbN,\ 0=s_0<s_1<\dots<s_n=1\right\}.
\]


\medskip

\subsection{From a sub-additive function $F:\Gamma\to \bbR_+$ to a CC-metric on $G_\infty$}
\label{sub:CC-construction4fun} \hfill{}\\
Consider a sub-additive function $F:\Gamma\to\bbR_+$ that is bi-Lipschitz to a word metric, i.e. satisfies
\begin{equation}\label{e:asym-dist}
	\begin{split}
		& F(\gamma_1\gamma_2)\le F(\gamma_1)+F(\gamma_2),\\
		&a\cdot d(e,\gamma) \le F(\gamma)\le b\cdot d(e,\gamma),
	\end{split}
\end{equation}
for some constants $0<a\le b<\infty$. 
Note that the upper linear bound $F(\gamma)\le b d(e,\gamma)$
follows automatically from subadditivity and the fact that $\Gamma$ is finitely generated;
so the content of the second assumption is the lower linear bound for $F:\Gamma\to\bbR_+$.

Such a function induces a subadditive function 
\[
	f:\ab{\Gamma}\to\bbR_+
\] 
using the following general construction.
\begin{lemma}\label{L:sub-to-quo}
	Let $1\to \Delta\to\Gamma\to\Lambda\to 1$ be a short exact sequence of groups,
	and $F:\Gamma\to \bbR_+$ a subadditive function. Then the function
	\[
		f:\Lambda\to \bbR_+\qquad\textrm{defined\ by}\qquad f(\gamma \Delta):=\inf\{ F(\gamma \delta) \mid \delta\in \Delta\}
	\]
	is subadditive.
\end{lemma}
\begin{proof}
	Given $\lambda_1,\lambda_2\in\Lambda$ and $\epsilon>0$ choose $\gamma_1,\gamma_2\in\Gamma$ so that
	$\lambda_i=\gamma_i\Delta$ and $F(\gamma_i)\le f(\lambda_i)+\epsilon$ for $i=1,2$.
	Then $\lambda_1\lambda_2=\gamma_1\gamma_2\Delta$, so
	\[
		f(\lambda_1\lambda_2)\le F(\gamma_1\gamma_2)\le F(\gamma_1)+F(\gamma_2)\le f(\lambda_1)+f(\lambda_2)+2\epsilon.
	\]
	Since $\epsilon>0$ is arbitrary we get $f(\lambda_1\lambda_2)\le f(\lambda_1)+f(\lambda_2)$.
\end{proof}

Now recall that $\Gamma$ is a uniform lattice in its Mal'cev completion $G$.
In fact, viewing $G$ as the $\bbR$-points $G=\mathbf{G}_\bbR$ of a $\bbQ$-algebraic group $\mathbf{G}$, 
we may think of $\Gamma$ as (commensurable to) $\mathbf{G}_\bbZ$.
Taking the abelianization is a $\bbQ$-algebraic operation, hence $\ab{\Gamma}$ is (commensurable to) 
$\ab{\mathbf{G}}_\bbZ$, a lattice in $\ab{\mathbf{G}}_\bbR=\ab{G}$.
Hence $\ab{\Gamma}$, that is abstractly isomorphic to $\bbZ^d$,
is a lattice in $\ab{G}$, that is continuously isomorphic to $\bbR^d$ with $d=\dim{\frak{v}_1}$.
One often writes
\[
	\ab{G}=\ab{\Gamma}\otimes \bbR
\]
to emphasize that $\ab{\Gamma}$ is a lattice in the real vector space $\ab{G}$.

\begin{lemma}\label{L:sub-on-Zd}
	Let $\Lambda$ be a lattice in a finite dimensional real vector space $V$, and $f:\Lambda\to\bbR_+$ be a subadditive function.
	Then there exists a unique homogeneous subadditive function $\phi:V\to\bbR_+$ so that $f$ is asymptotically equivalent to 
	$\phi|_\Lambda$; in particular
	\[
		\phi(\lambda)=\lim_{n\to\infty} \frac{1}{n}f(n\lambda)=\inf_{n\ge 1}\frac{1}{n}f(n\lambda).
	\]
	Moreover, if $c_1\le f(\lambda)/\|\lambda\|\le c_2$ on $\Lambda$, then
	$c_1\le \phi(v)/\|v\|\le c_2$ on $V\setminus\{0\}$.
\end{lemma} 

This is an easy and well known fact; but see Burago's \cite{Bu} for much finer results in case of a coarsely geodesic metric.

\begin{remark}\label{R:inner4abelian}
	It follows that any subadditive function $f:\bbZ^d\to \bbR_+$ is automatically \textbf{inner} in the following sense:
	given $\epsilon>0$ there is $R<\infty$ so that any $\lambda\in\bbZ^d$ can be written as $\lambda=\lambda_1+\dots+\lambda_n$
	with 
	\[
		f(\lambda_i)\le R\quad(i=1,\dots,n),\qquad f(\lambda_1)+\dots+f(\lambda_n)\le (1+\epsilon)\cdot f(\lambda).
	\]
	Indeed, this is clear for the asymmetric norm $\phi:\bbR^d\to\bbR_+$ associated with $f$ in Lemma~\ref{L:sub-on-Zd},
	and translates to $f$ by the virtue of the approximation.
\end{remark}

\medskip

\begin{lemma}\label{L:turning-to-horizontal}
	Let $F:\Gamma\to \bbR_+$ be a subadditive function, $f:\ab{\Gamma}\to\bbR_+$ and $\phi:\ab{\Gamma}\otimes \bbR\to\bbR_+$ 
	be defined by Lemmas~\ref{L:sub-to-quo} and \ref{L:sub-on-Zd}. Then for any $\gamma\in\Gamma$ one has
	\[
		\lim_{n\to\infty}\ \frac{1}{n} F(\gamma^n)\ =\ \inf_{n\ge 1}\ \frac{1}{n} F(\gamma^n)
		\ =\ \lim_{n\to\infty}\ \frac{1}{n} f((\ab{\gamma})^n)\ =\ \inf_{n\ge 1}\ \frac{1}{n} f((\ab{\gamma})^n)
		\ =\ \phi(\ab{\gamma})
	\]
	for any $\gamma\in\Gamma$ with $\ab{\gamma}=\gamma [\Gamma,\Gamma]$ denoting the image in $\ab{\Gamma}$.
\end{lemma}
\begin{proof}
	The sequence $a_n=F(\gamma^n)$ is sub-additive, i.e. $a_{n+m}\le a_n+a_m$ for all $n,m\in\bbN$; 
	hence $a_n/n$ converges to $\inf a_n/n$.
	Next we note that Lemmas~\ref{L:flattening-of-powers} and relation (\ref{e:lengthscaling}) imply that
	whenever $\gamma_1,\gamma_2\in\Gamma$ satisfy $\ab{\gamma}_1=\ab{\gamma}_2$ one has 
	\[
		\lim_{n\to\infty}\frac{1}{n}\cdot d(\gamma_1^n,\gamma_2^n)=0.
	\]
	Since any sub-additive function is automatically Lipschitz with respect to the word metric,
	it follows that $\lim F(\gamma_1^n)/n=\lim F(\gamma_2^n)/n$. 
	Thus this limit of $F(\gamma^n)/n$ depends only on $\ab{\gamma}$, and is easily seen to be
	$\lim f((\ab{\gamma})^n)/n$, i.e. $\phi(\ab{\gamma})$.
\end{proof}

\medskip

%

We can now apply the Carnot-Carath\'eodory construction to define a (possibly asymmetric) metric $d_\phi$ on $G_\infty$ by
\begin{equation}\label{e:dphi}
	d_\phi(g,g'):=\inf\left\{ \len_\phi(\xi) \left|\ \xi\textrm{\ is\ an\ admissible\ curve\ from\ }g\ \textrm{to\ } g' \right.\right\}.
\end{equation}

\medskip

\subsection{From a cocycle $c:\Gamma\times X\to \bbR_+$ to the CC-metric} 
\label{sub:cocycle2cc}\hfill{}\\
Let $\Gamma$ be a finitely generated, virtually nilpotent group, $\Gamma\acts (X,m)$ an ergodic p.m.p. action, 
and $c:\Gamma\times X\to \bbR_+$ a subadditive cocycle with $c(\gamma,-)\in L^\infty(X,m)$ for every $\gamma\in\Gamma$. 
We start with a couple of remarks about passing to finite index subgroups and dividing by finite kernels. 

Let $\Gamma'<\Gamma$ be a subgroup of finite index. The action of $\Gamma'$ on $(X,m)$ has at most $[\Gamma:\Gamma']$-many ergodic components 
permuted by the $\Gamma$-action. Let $c':\Gamma'\times X'\to\bbR_+$ be the restriction of $c$ to one of the $\Gamma'$-ergodic components $X'\subset X$.
If one shows that there is some function $\Phi:G_\infty\to\bbR_+$ so that for a.e. $x'\in X'$ the function $c'(-,x'):\Gamma'\to \bbR_+$ is asymptotically equivalent to $\Phi$,
then the same would apply to $c(-,x):\Gamma\to \bbR_+$ for a.e. $x\in X$.
Indeed choosing representatives $\gamma_1,\dots,\gamma_n$ for $\Gamma'\backslash \Gamma$ for every $\gamma\in\Gamma$ one can write
\[
	c(\gamma,x)= c(\gamma'\gamma_i,x)\le c(\gamma',\gamma_i.x)+\|c(\gamma_i,x)\|_\infty
\]
for some $\gamma'\in\Gamma'$; similarly 
\[
	c(\gamma',x)= c(\gamma\gamma_i^{-1},x)\le c(\gamma,\gamma_i^{-1}.x)+\|c(\gamma_i^{-1},x)\|_\infty.
\]
Hence $c(\gamma,x)$ is at uniformly bounded distance from $c(\gamma',\gamma_i^{\pm 1}.x)$, and therefore has the same
asymptotic behavior.

Let $N$ be a finite normal subgroup of $\Gamma$. Then $\Gamma_1:=\Gamma/N$ acts ergodically by p.m.p. transformations on $(X_1,m_1):=(X,m)/N$.
A subadditive cocycle $c:\Gamma\times X\to \bbR_+$ defines $c_1:\Gamma_1\times X_1\to\bbR_+$ by
\[
	c_1(\gamma_1, x_1):=\max \{ c(\gamma,x) \mid \pr(\gamma)=\gamma_1,\ \pr(x)=x_1\}.
\] 
Then $c_1:\Gamma_1\times X_1\to\bbR_+$ is a sub-additive cocycle, and it is within bounded distance from $c(\gamma,x)$.

Furthermore we note that conditions (i) and (ii) of Theorem~\ref{T:NilKingman} pass to
$c'$ and $c_1$ as above. (Condition (ii) for $c'$ is an easy exercise using subadditivity and innerness; the others are immediate.)
Hence in the context of Theorem~\ref{T:NilKingman} (and Theorem~\ref{T:REM}) 
we may assume without loss of generality that $\Gamma$ itself is finitely-generated, torsion-free, nilpotent group 
with torsion-free abelianization $\ab{\Gamma}$. 
Hereafter we shall make this assumption.

Let us define the function 
\[
	\overline{c}:\Gamma\to \bbR_+\qquad\textrm{by\ setting}\qquad \overline{c}(\gamma):=\int_X c(\gamma,x)\,dm(x).
\]
Observe that $\overline{c}$ is a sub-additive function, because sub-additivity of $c$ and $\Gamma$-invariance of $m$ imply
\[
	\overline{c}(\gamma_1\gamma_2)=\int_X c(\gamma_1\gamma_2,x)\,dm(x)\le 
	\int_X c(\gamma_1,\gamma_2x)\,dm(x)+\int_X c(\gamma_2,x)\,dm(x)=\overline{c}(\gamma_1)+\overline{c}(\gamma_2).
\]
Moreover one always has an upper linear estimate
\[
	\overline{c}(\gamma)\le K_1\cdot |\gamma|_S
	\qquad\textrm{with}\qquad
	K_1:=\max\{ \overline{c}(s) \mid s\in S\}.
\]
The definition of $\overline{c}$ requires only $L^1$-integrability of the functions $c(\gamma,x)$.
We note the point-wise bi-Lipschitz condition (i) passes to the average, and we have 
\begin{equation}\label{e:cbar-biLip}
	k\cdot |\gamma|_S\le \overline{c}(\gamma)\le K_1\cdot |\gamma|_S
\end{equation}
with constants $0<k\le K_1<+\infty$ and any $\gamma\in\Gamma$.
\begin{remark}
	It does not seem to be obvious why the condition of being \emph{inner} for the sub-additive cocycle $c:\Gamma\times X\to\bbR_+$ 
	(condition (ii) in Theorem~\ref{T:NilKingman}) should imply innerness for the average 
	sub-additive function $\overline{c}:\Gamma\to\bbR_+$.
	It will follow from our results that for an $L^\infty$-cocycle $c$ over an \emph{ergodic} $\Gamma$-action the average 
	$\overline{c}$ is indeed inner as it is asymptotically equivalent to a Carnot-Carath\'eodory function $\Phi$.
\end{remark}
We can now summarize the construction
\begin{prop}\label{P:construction}
	Let $\Gamma\acts (X,m)$ and $c:\Gamma\times X\to\bbR_+$	be a subadditive cocycle satisfying 
	condition (i) in Theorem~\ref{T:NilKingman}. Then:
	\begin{itemize}
		\item 
		The average function
		\[
			\overline{c}(\gamma):=\int_X c(\gamma,x)\,dm(x)
		\]
		is a subadditive function on $\Gamma$, satisfying bi-Lipschitz condition (\ref{e:cbar-biLip}).
		\item
		This subadditive function defines a subadditive, homogeneous $\phi:\ab{\Gamma}\otimes \bbR\to \bbR_+$, such that
		\[
			\lim_{n\to\infty}\frac{1}{n}\overline{c}(\gamma^n)=\phi(\ab{\gamma})\qquad (\gamma\in\Gamma).
		\]
		Moreover, for some $0<a\le b<\infty$ one has $a\cdot\|v\|\le \phi(v)\le b\cdot \|v\|$ 
		for all $v\in \ab{\Gamma}\otimes\bbR\cong \ab{\mfg}_\infty$.
		\item
		The Carnot-Carath\'eodory construction defines an asymmetric distance on $G_\infty$
		\[
			d_\phi:G_\infty\times G_\infty\to\bbR_+
		\] 
		that is right-invariant, homogeneous, and bi-Lipschitz to $d_\infty$ as in (\ref{e:dphi-properties}).
		We denote
		\[
			\Phi(g):=d_\phi(e,g)\qquad (g\in G_\infty).
		\] 
	\end{itemize}
\end{prop}



\section{Preparation for the main proofs} 
\label{sec:prep}

In this section we prepare two tools for the proof of the main results.
The first tool is a purely geometric fact that allows one to approximate 
an admissible curve in the asymptotic cone $G_\infty$ of $\Gamma$
by rescaled sequences of the form $T_k^n\cdots T_2^nT_1^n$ in $\Gamma$; we call such sequences \emph{polygonal paths}.
The second tool is an ergodic theorem  for a sub-additive
cocycle \emph{along polygonal paths} over a general ergodic, p.m.p. action of a nilpotent group.

\medskip

\subsection{Approximating curves in $G_\infty$ by polygonal paths in $\Gamma$} 
\label{sub:subsection_name}\hfill{}\\
This subsection concerns purely geometric aspects of the convergence of $\Gamma$ to its asymptotic cone $G_\infty$
(and is unrelated to the action $\Gamma\acts (X,m)$ and the cocycle $c:\Gamma\times X\to \bbR_+$).

As before, $\Gamma$ is a finitely generated, torsion-free, nilpotent group with torsion-free abelianization, $d$ is a right-invariant word metric on $\Gamma$,
$(G_\infty,d_\infty)$ is the asymptotic cone, and  
\[
	\scl{t}{-}:\Gamma\overto{} G_\infty\qquad (t>0)
\]
are the maps defined in (\ref{e:sclt}) that realize the Gromov-Hausdorff convergence 
\[
	(\Gamma,\frac{1}{t}\cdot d,e_\Gamma)\ \overto{}\ (G_\infty,d_\infty,e).
\]
We also fix a (possibly) asymmetric norm 
\[
	\phi:\ab{\mfg}_\infty\ \overto{}\ \bbR_+
\]
satisfying (\ref{e:asym-norm}) and use it to associate length $\len_\phi(\xi)$ to admissible curves $\xi:[0,1]\to G_\infty$.
We denote the balls in $G_\infty$ by
\[
	\Ball(g,\epsilon):=\{g'\in G_\infty \mid d_\infty(g,g')<\epsilon\}.
\]
\begin{prop}[Approximation of curves by polygonal paths]\label{P:curve-to-polygon}\hfill{}\\
	Given a Lipschitz curve $\xi:[0,1]\to G_\infty$ with $\xi(0)=e$, 
	and $\epsilon>0$ one can find $k,p,n_0\in\bbN$, $T_1,\dots,T_k\in\Gamma$
	so that for $n\ge n_0$ one has:
	\[
		\sum_{j=1}^k\ d_\infty\left(\ssc{np}{T_j^n\cdots T_2^nT_1^n},\,\xi(\frac{j}{k})\right)<\epsilon,
	\]
	and
	\[
		\left|\frac{1}{p}\cdot\left(\phi(\ab{T}_k)+\dots+\phi(\ab{T}_1)\right) - \len_\phi(\xi)\right|<\epsilon.
	\]
\end{prop}
We emphasize the order of the main quantifiers: 
the elements $T_1,\dots,T_k$ and $p\in\bbN$ depend only on the required accuracy $\epsilon>0$ (and of course the curve $\xi$), 
and provide $\epsilon$-good approximation at \emph{all sufficiently large scales}. 

We shall need this proposition (in combination with Theorem~\ref{T:erg-poly}) in two cases:
\begin{itemize}
	\item 
	In \S~\ref{sub:upperbound} we choose $\xi$ to be a $\phi$-geodesic connecting $e$ to some $g$.
	In this case $\xi$ is a smooth admissible curve
	and we are interested in the inequality 
	\[
		\frac{1}{p}\cdot\left(\phi(\ab{T}_k)+\dots+\phi(\ab{T}_1)\right)\le \len_\phi(\xi)+\epsilon=\Phi(g)+\epsilon
	\]
	while $d_\infty(\ssc{np}{T_k^n\cdots T_2^nT_1^n},g)<\epsilon$. 
	\item 
	In \S~\ref{sub:lowerbound} we get a Lipschitz curve $\xi$ connecting $e$ to some $g$. 
	In this case we are interested in the inequality
	\[
		\frac{1}{p}\cdot\left(\phi(\ab{T}_k)+\dots+\phi(\ab{T}_1)\right)\ge \len_\phi(\xi)-\epsilon\ge \Phi(g)-\epsilon
	\]
	while requiring  
	\[
		\sum_{j=1}^k\ d_\infty\left(\ssc{np}{T_k^n\cdots T_2^nT_1^n},\xi(\frac{j}{k})\right)<\epsilon
	\]
	which is stronger than just $d_\infty(\ssc{np}{T_k^n\cdots T_2^nT_1^n},g)<\epsilon$.
\end{itemize}

\medskip

\begin{proof}[Proof of Proposition~\ref{P:curve-to-polygon}]\hfill{}\\
	First we work in $G_\infty$. Our goal is to find $k\in\bbN$ and horizontal vectors 
	\[
		v_1,\dots,v_k\in \frak{v}_1\subset \mfg_\infty
	\]
	so that, denoting $h_j:=\exp_\infty(v_j)$ one has
	\begin{equation}\label{e:smooth-poly}
		\sum_{j=1}^k\ d_\infty(h_j\cdots h_1,\,\xi(\frac{j}{k}))<\half\epsilon,
		\qquad |\sum_{j=1}^k \phi(v_j)-\len_\phi(\xi)|<\half\epsilon.
	\end{equation}
	For a fixed $k\in\bbN$, that we will take to be sufficiently large, we define $v_1,\dots,v_k$ inductively as follows:
	set 
	\[
		v_1:=\pi_\infty\circ \log_\infty(\xi(\frac{1}{k})),\qquad h_1:=\exp_\infty(v_1).
	\]
	Assuming $v_1,\dots,v_{j-1}$ were chosen, set
	\[
		v_{j}:=\pi_\infty\circ \log_\infty(\xi(\frac{1}{k})(h_{j-1}\cdots h_1\xi(0))^{-1}),\qquad h_j:=\exp_\infty(v_j).
	\]
	Here $\pi_\infty:\mfg_\infty\to \frak{v}_1$ is the linear projection corresponding to the decomposition 
	$\mfg_\infty=\oplus_{i=1}^r \frak{v}_i$.
	
	Let us now show that by choosing $k$ large enough we can guarantee (\ref{e:smooth-poly}).
	To this end we need the fact (\cite[Lemme (18)]{Pansu}) that in the unit ball in $G_\infty$
	the "horizontal component" gives an approximation with at most quadratic error. 
	More precisely, there is a constant $C_1$ so that for all $g\in \Ball(e,1)$:
	\[
		d_\infty(g,\exp_\infty\circ \pi_\infty\circ \log_\infty (g))\le C_1\cdot d_\infty(e,g)^2.
	\]
	Hence for large $k$ one has for $j=1,\dots,k$:
	\[
		\begin{split}
			&d_\infty\left(h_j\cdots h_1,\,\xi(\frac{j}{k})\right)
			\le C_1\cdot d_\infty\left(h_{j-1}\cdots h_1,\xi(\frac{j}{k})\right)^2\\
			&\le C_1\cdot \left(d_\infty\left(h_{j-1}\cdots h_1,\xi(\frac{j-1}{k})\right)
				+d_\infty\left(\xi(\frac{j-1}{k}),\xi(\frac{j}{k})\right)\right)^2
			\le  C_2\cdot \frac{1}{k^2}
		\end{split}
	\]
	for some $C_2$	depending on $C_1$ and the Lipschitz constant of $\xi$.
	Hence for all $k$ large enough
	\[
		\sum_{j=1}^k\ d_\infty\left(h_j\cdots h_1,\,\xi(\frac{j}{k})\right)
		<C_2\cdot k\cdot \frac{1}{k^2}=\frac{C_2}{k}<\half\epsilon.
	\]
	The second fact that we want to use is that a Lipschitz curve $\xi:[0,1]\to G_\infty$ is rectifiable.
	Therefore 
	\[
		\len_\phi(\xi)=\lim_{k\to\infty} \sum_{j=1}^k d_\phi\left(\xi(\frac{j-1}{k}), \xi(\frac{j}{k})\right).
	\]
	One also has a constant $C$ so that
	\[
		\left|d_\phi(g,g')-\phi\circ\pi_\infty\circ \log_\infty(g'g^{-1})\right|\le C\cdot d_\infty(g,g')^2
	\]
	whenever $g'\in \Ball(g,1)$. Thus for all sufficiently large $k$ and for each $j=1,\dots,k$, we have
	\[
		\begin{split}
			&\left|d_\phi\left(\xi(\frac{j-1}{k}), \xi(\frac{j}{k})\right)-\phi(v_j)\right|
					\le  \left|d_\phi\left(h_{j-1}\cdots h_1, \xi(\frac{j}{k})\right)-\phi(v_j)\right|\\
			&\qquad		+ d_\phi\left(h_{j-1}\cdots h_1,\xi(\frac{j-1}{k})\right)
					\le C_3 \cdot \frac{1}{k^2}
		\end{split}
	\]
	and the second inequality in (\ref{e:smooth-poly}) follows.
	
	We have now found $k\in\bbN$ and horizontal vectors $v_1,\dots,v_k\in\frak{v}_1$ satisfying (\ref{e:smooth-poly}),
	and need to find $T_1,\dots, T_k\in\Gamma$, $p\in \bbN$, and $n_0>0$ as in the Proposition.
	We need the following
	\begin{lemma}\label{L:angular-approx}\hfill{}\\
		Given a horizontal vector $v\in \frak{v}_1 < \mfg_\infty$ and $\epsilon'>0$ there exist $\tau\in\Gamma$, $p\in\bbN$ 
		and $n_0$ so that  
		\[
			|\frac{1}{p}\phi(\ab{\tau})-\phi(v)|<\epsilon',\qquad 
			d_\infty(\ssc{np}{\tau^n},\exp_\infty(v))<\epsilon'\qquad (n>n_0)
		\]
		where $\exp_\infty:\mfg_\infty\to G_\infty$ is the exponential map on $G_\infty$.
	\end{lemma}
	\begin{proof}
		Since $\ssc{p}{\Gamma}$ becomes denser and denser in $G_\infty$ as $p\to\infty$, 
		one can find $p\in\bbN$ and $\gamma\in\Gamma$ so that
		$\ssc{p}{\gamma}$ is close to $\exp_\infty(v)$.
		Recall that 
		\[
			\ssc{p}{\gamma}=\exp_\infty\left(\frac{1}{p}\cdot L\circ\pi_1\circ\log(\gamma)
				+\frac{1}{p^2}\cdot L\circ\pi_2\circ\log(\gamma)+\dots
				+\frac{1}{p^r}\cdot L\circ\pi_r\circ\log(\gamma)\right)
		\] 
		where $\pi_j:\mfg\to V_j=L^{-1}(\frak{v}_j)$ are the linear projections.
		Since $v\in\frak{v}_1=L(\pi_1(\mfg))$, it follows that $p^{-1}\cdot L\circ\pi_1\circ\log(\gamma)$ 
		and $v$ are close.
		Hence we may choose $p$ and $\gamma$ (to be called $\tau$)
		so that  
		\[
			d_\infty\left(\exp_\infty \left(\frac{1}{p}\cdot L\circ \pi_1\circ\log(\gamma)\right),\,
			\exp_\infty(v)\right)<\epsilon',
			\qquad
			\left|\phi(\frac{1}{p}\cdot L\circ \pi_1\circ\log(\gamma))-\phi(v)\right|<\epsilon'. 
		\] 
		Now considering powers $\ssc{n}{\gamma^n}$ we are done by applying Lemma~\ref{L:flattening-of-powers}.
		This proves Lemma~\ref{L:angular-approx}.
	\end{proof}
	
	Choose $\epsilon'\in(0,\epsilon/2k)$ small enough to ensure that whenever $h_1',\dots,h_k'\in G_\infty$ 
	are $\epsilon'$-close to $h_1,\dots,h_k$, respectively, one has 
	\[
		\sum_{j=1}^k\ d_\infty\left(h'_j\cdots h'_2h'_1,h_j\cdots h_2h_1\right)<\half\epsilon.
	\]
	Let us now apply Lemma~\ref{L:angular-approx} with $\epsilon'>0$ as above
	to obtain elements $\tau_1,\dots,\tau_k$ and $p_1,\dots,p_k\in\bbN$
	so that the pairs $(\tau_j, p_j)$ satisfy
	\[
		|\frac{1}{p_j}\phi(\ab{\tau}_j)-\phi(v_j)|<\epsilon'<\frac{\epsilon}{2k},
			\qquad
		d_\infty\left(\ssc{np_j}{\tau_j^n},h_j\right)<\epsilon'\qquad (j=1,\dots,k).
	\]
	Replacing a pair $(\tau_j,p_j)$ by $(\tau_j^q,q\cdot p_j)$ with any $q\in\bbN$, the above inequalities 
	clearly remain valid.
	So taking $p:=p_1\cdots p_k$ and replacing $(\tau_j, p_j)$ by $(T_j:=\tau_j^{p/p_j}, p)$ we get
	elements $T_1,\dots,T_k\in \Gamma$ so that for $n\gg1$
	\[
			d_\infty\left(\ssc{np}{T_j^n},h_j\right)<\epsilon'\qquad
			\textrm{and}\qquad |\frac{1}{p}\phi(\ab{T}_j)-\phi(v_j)|<\epsilon'<\frac{\epsilon}{2k}
			\qquad (j=1,\dots,k).
	\]
	In view of Lemma~\ref{L:multiplication} we know that for every $j=1,\dots,k$
	\[
		d_\infty\left(\ssc{np}{(T_j^n\cdots T_1^n)},(\ssc{np}{T_j^n})\cdots (\ssc{np}{T_1^n})\right)\to 0.
	\]
	Thus for all $n$ large enough, we have
	\[
		\sum_{j=1}^k\ d_\infty\left(\ssc{np}{T_j^n\cdots T_1^n},h_j\cdots h_1\right)<\half\epsilon,
	\]
	while 
	\[
		\left|\frac{1}{p}\cdot\sum_{j=1}^k \phi(\ab{T}_j) -\sum_{j=1}^k \phi(v_j)\right|<\half\epsilon.
	\]
	Combined with (\ref{e:smooth-poly}) this establishes the required inequalities. 
	This completes the proof of Proposition~\ref{P:curve-to-polygon}.
\end{proof}

\medskip

\medskip

\subsection{An Ergodic Theorem along polygonal paths} 
\label{sub:ergodic}\hfill{}\\
The goal of this subsection is to prove the following result that might have an independent interest.

\begin{theorem}[Ergodic theorem along polygonal paths]\label{T:erg-poly}\hfill{}\\
	Let $\Gamma$ be a finitely generated torsion-free nilpotent group with torsion-free abelianization, $\Gamma\acts (X,m)$ an ergodic p.m.p. action, $c:\Gamma\times X\to \bbR_+$
	a measurable, non-negative, subadditive cocycle with $c(\gamma,-)\in L^\infty(X,m)$ for every $\gamma\in\Gamma$,
	and $\overline{c}$ and $\phi$ as above.
	Then for any $T_1,\dots,T_k\in \Gamma$ one has $m$-a.e. and $L^1(X,m)$-convergence
	\[
		\lim_{n\to \infty} \frac{1}{n}\cdot c(T_j^n,\, T_{j-1}^n \cdots T_1^n x)=\phi(\ab{T}_j)
	\]
	for each $j=1,\dots,k$, and consequently
	\[
		\lim_{n\to \infty} \frac{1}{n}
			\left(c(T_k^n,\, T_{k-1}^n \cdots T_1^n x)+\dots+c(T_2^n,\, T_{1}^nx)+c(T_1^n,x)\right)
			=\phi(\ab{T}_k)+\cdots +\phi(\ab{T}_1).
	\]
	The $L^1$ convergence holds under a weaker assumption: $c(\gamma,-)\in L^1(X,m)$.
\end{theorem}

\medskip

The case $k=1$ was shown by Austin \cite{Austin} under the weaker assumption that $c(\gamma,-)\in L^1(X,m)$ for every $\gamma\in\Gamma$.
For reader's convenience we include a proof.

\begin{theorem}[Austin \cite{Austin}]\label{L:k-1}\hfill{}\\
	Let $c:\Gamma\times X\to\bbR_+$ be a subadditive cocycle with $c(\gamma,-)\in L^1(X,m)$ for every $\gamma\in\Gamma$.
	Then for any $T\in\Gamma$ one has
	\[
		\lim_{n\to\infty}\frac{1}{n} c(T^n,x)=\phi(\ab{T})
	\] 
	for $m$-a.e. $x\in X$ and in $L^1(X,m)$. 
\end{theorem}
\begin{proof}
		Kingman's subadditive ergodic theorem, applied to the sub-additive cocycle $h_n(x):=c(T^n,x)$ over $(X,m,T)$,
		gives an $m$-a.e. and $L^1$ convergence
		\[
			\lim_{n\to\infty} \frac{1}{n} c(T^n,x) = h(x),
		\]
		where $h(x)$ is a measurable $T$-invariant function, satisfying
		\[
			\int_X h(x)\,dm(x)=\lim_{n\to\infty}\frac{1}{n}\cdot\int_X h_n(x)\,dm(x)=\lim_{n\to\infty}\frac{1}{n}\overline{c}(T^n)=\phi(\ab{T}).
		\]
		(We used Lemma~\ref{L:turning-to-horizontal} in the last equality).
		Fix $\gamma\in \Gamma$ and denote $\gamma_n:=T^n\gamma T^{-n}$. 
		Since $\gamma_n T^n=T^n\gamma$ we have
		\begin{equation}\label{e:gammanTn}
			c(T^n,x)-c(\gamma_n^{-1},\gamma_nT^n.x)\le c(\gamma_nT^n,x)=c(T^n\gamma,x)\le c(T^n,\gamma.x)+c(\gamma,x).
		\end{equation}
		Denote $f_n(x)=n^{-1}(c(\gamma_n,x)+c(\gamma_n^{-1},\gamma_nT^n.x))$
		and observe that since one has $|\gamma^{-1}_n|_S=|\gamma_n|_S=o(n)$ (cf. Breuillard \cite[Lemma 5.6]{Bre})
		\[
			\|f_n\|_1\le \frac{K_1}{n}\cdot 2|\gamma_n|_S\to 0,\qquad \textrm{where}\qquad K_1=\max_{s\in S} \|c(s,-)\|_1.
		\]
		Thus there is a sequence $n_i\to\infty$ so that $f_{n_i}(x)\to 0$ for $m$-a.e. $x\in X$.
		Dividing (\ref{e:gammanTn}) by $n$, and taking the limit along the subsequence $n_i$, one obtains 
		\[
			h(x)\le	h(\gamma.x).
		\]
		Since this is true a.e. for every $\gamma\in\Gamma$, $h$ is $\Gamma$-invariant. By ergodicity it is constant. 
		This constant is $\phi(\ab{T})$ by integration.
\end{proof}

In the general case of $k\ge 2$ the term $n^{-1}\cdot c(T_1^n,x)$ converges to $\phi(\ab{T}_1)$ by the above, but
dealing with the next terms, such as $n^{-1}\cdot c(T_2^n,T_1^n.x)$, one faces a "moving target" problem.
We shall overcome this difficulty by finding regions $Z_2,\dots,Z_k\subset X$ where 
$n^{-1}c(T_\ell^n,z)=\phi(\ab{T}_\ell)+o(n)$ for $z\in Z_\ell$, and perturbing the polygonal path $T_k^n\cdots T_2^nT_1^n$
slightly to make sure to land in the appropriate regions at appropriate times.
We need several lemmas.

\medskip

\begin{lemma}[Parallelogram inequality]\label{L:4points}\hfill{}\\
	Given $\alpha,\beta,\tau,\tau'\in\Gamma$ one has
	\[
		\left|c(\tau,\,\alpha.x)-c(\tau',\,\beta.x)\right| \le K\cdot (d(\alpha,\beta)+d(\tau\alpha,\tau'\beta))
	\]
	for $K=\max_{s\in S} \|c(s,-)\|_\infty$.
\end{lemma}
\begin{proof}
	Let us write $\beta=\delta\alpha$ and $\tau'\beta=\omega\tau\alpha$, so 
	\[
		|\delta|=|\delta^{-1}|=d(\alpha,\beta),
		\qquad
		|\omega|=|\omega^{-1}|=d(\tau\alpha,\tau'\beta).
	\]
	Since $\tau'=\omega\tau\delta^{-1}$ we have
	\[
		\begin{split}
			c(\tau',\beta.x)&= c(\omega\tau\delta^{-1},\beta.x)
			\le c(\omega,\tau\alpha.x)+c(\tau,\alpha.x)+c(\delta^{-1},\beta.x)\\
			&\le c(\tau,\alpha.x)+K\cdot (|\omega|+|\delta^{-1}|).
		\end{split}
	\]
	Conversely
	\[
		\begin{split}
			c(\tau,\alpha.x)&= c(\omega^{-1}\tau'\delta,\alpha.x)
			\le c(\omega^{-1},\tau'\beta.x)+c(\tau',\beta.x)+c(\delta,\alpha.x)\\
			&\le c(\tau',\beta.x)+K\cdot (|\omega^{-1}|+|\delta|).
		\end{split}
	\]
\end{proof}

\medskip

\begin{lemma}\label{L:moving-target}
	Let $T\in\Gamma$, $\delta>0$, and a measurable subset $E\subset X$ be given. Then
	the set 
	\[
		E^*:=\left\{ x\in X \mid \liminf_{n\to\infty}\frac{\#\{ n'<\delta\cdot n \mid T^{n-n'}.x\in E\}}{\delta\cdot n}>0\right\}
	\]
	has $m(E^*)\ge m(E)$. Moreover, given $\epsilon>0$ there is $N$ so that the set
	\[
		E^*_N:= \left\{ x\in X \mid \forall n\ge N,\ \frac{\#\{ n'<\delta\cdot n \mid T^{n-n'}.x\in E\}}{\delta\cdot n}>0\right\}
	\]
	has $m(E^*_N)>m(E)-\epsilon$.
\end{lemma}
\begin{proof}
	Given a function $f\in L^1(X,m)$ and integers $1\le k<n$ consider the averaged function
	\[
		A_k^nf(x):=\frac{1}{n-k}\cdot\sum_{j=k}^{n-1} f(T^j.x).
	\]
	Birkhoff's pointwise ergodic theorem asserts $m$-a.e. convergence
	\[
		\lim_{n\to\infty} A_{0}^nf =\bbE(f \mid \mathcal{B}^T)
	\]
	to the conditional expectation of $f$ with respect to the the sub-$\xi$-algebra of $T$-invariant sets $\mathcal{B}^T$.
	(The conditional expectation is defined only up to null sets, but so is the above convergence). 
	We observe that since 
	\[
		A_0^nf(x) = \frac{k}{n} \cdot A_0^kf(x) + \frac{n-k}{n} \cdot A_{k}^nf(x),
 	\]
	taking $k=\lceil(1-\delta)n\rceil$ with $0<\delta<1$ fixed and letting $n\to\infty$, 
	it follows that for $m$-a.e. $x\in X$
	\[
		\frac{1}{\eta\cdot n} \sum_{\lceil(1-\eta)n\rceil}^n f\circ T^j\ \overto{}\ \bbE(f \mid \mathcal{B}^T). 
	\]
	Applying this to the characteristic function $f=1_E$ of $E\subset X$, we deduce
	that for $m$-a.e. $x\in X$
	\[
		\lim_{n\to\infty}\frac{\#\{ \lceil (1-\delta)n\rceil\le j\le n \mid T^j.x\in E\}}{\delta\cdot n}=h_E(x),
	\]
	where $h_E:=\bbE(E \mid \mathcal{B}^T)$ is the the conditional expectation of $1_E$.
	Since $0\le h_E(x)\le 1$ a.e. while $\int h_E=m(E)$, it follows that the set $\{ x\in X \mid h_E(x)>0\}$ has measure $\ge m(E)$.
	Yet the set $\{x \mid h_E(x)>0\}$ is, up to null sets, precisely $E^*$. Hence $m(E^*)\ge m(E)$.

	For the second statement, note that $\{E^*_N\}$ is an increasing sequence of measurable sets whose union (=limit) 
	is $E^*$.
\end{proof}
\medskip

\begin{lemma}[Small perturbations of polygonal paths]\label{L:pert}\hfill{}\\
	Given $T_1,\dots,T_k\in\Gamma$ and $\epsilon>0$, there is $\delta>0$ and $N$ so that
	for all $n\ge N$ we have:
	\[
		\frac{1}{n}\cdot d(T_k^{n-n_k}T_2^{n-n_2}T_1^{n-n_1},\, 
		T_k^nT^n_{k-1}\cdots T^n_2T^n_1)<\epsilon
	\]
	for any $0\le n_1,\dots,n_k\le \delta\cdot n$.
\end{lemma}
This Lemma can also be shown by rescaling and passing to the Gromov-Hausdorff limit in $G_\infty$ and relying on Lemma~\ref{L:multiplication}.
Here we give a more direct argument.
\begin{proof}
	It suffices to show that for fixed $k\in\bbN$, $T,T_1,\dots,T_k\in\Gamma$, $\epsilon'>0$, there is $\delta>0$ so that 
	for $n\gg 1$ one has:
	\begin{equation}\label{e:push1}
		d(T_k^n\cdots T_2^n T_1^n T^{-m},T_k^n\cdots T_2^n T_1^n)<\epsilon'\cdot n\qquad (\forall m<\delta\cdot n).
	\end{equation}
	Indeed, applying such an argument to $T_{j+1},\dots,T_\ell$ and $T=T_{j}$ with $\epsilon'=\epsilon/k$ we get 
	\[
		\frac{1}{n}(T_\ell^n\cdots T_{j+1}^n T_j^{n-n_j}T_{j-1}^{n-n_{j-1}}\cdots T_1^{n-n_1}, 
		T_\ell^n\cdots T_{j+1}^n T_j^{n}T_{j-1}^{n-n_{j-1}}\cdots T_1^{n-n_1})<\epsilon',
	\]
	and summing these inequalities over $j=1,\dots,\ell-1$, we get the 
	estimate $\ell\cdot \epsilon'\le \epsilon$ as required.
	
	To establish (\ref{e:push1}) use the general group-theoretic identity $\ ba=a [a^{-1},b] b\ $
	to push terms from the right to the left creating some commutator factors.
	More precisely  
	\[
		\begin{split}
			&T_k^n\cdots T_2^n T_1^n T^{-m}=T_k^n\cdots T_2^n\cdot (T^{-m}\cdot [T^{m},T_1^n])\cdot T_1^{n}\\
			&=T_k^n\cdots T_3^n\, (T^{-m}\cdot [T^m,T_2^n]\cdot [T^m,T_1^n]\cdot [[T^m,T_1^n]^{-1},T_2^n])\,  T_2^nT_1^{n}=\dots\\
			&=(T^{-m}[T^m,T_k^n]\cdots [T^m,T_1^n]\cdots )\cdot T_k^n\cdots T_2^n T_1^n 
		\end{split}
	\]
	where the expression in the parentheses is a product of $O(k)$ factors each being a higher commutator of the form 
	\[
		[\cdots[[T^m,T_{j_1}^n]^{-1},T_{j_2}^n]\cdots,T_{j_s}^n].
	\]
	We need to show that the word length of the expression in parentheses is $<\epsilon n$, and it suffices to show that
	each of the $O(k)$-commutator expressions has length $<\epsilon'n$, where $\epsilon'$ depends on $\epsilon$ and $k$.
	Iterated commutators of order $s$ above the nilpotency degree $r$ give identity. For $s\le r$ one has 
	(cf. \cite[Lemma 3.8]{Bre})
	\begin{equation}\label{e:comm}
		[\cdots[T^m,T_{j_1}^n]^{-1},\cdots,T_{j_s}^n]=[\cdots[T,T_{j_1}],\cdots,T_{j_s}]^{\pm m\cdot n^s}
	\end{equation}
	For each one of the finitely many elements $\gamma=[\cdots[T,T_{j_1}],\cdots,T_{j_s}]$ as above, we have 
	\[
		|\gamma^{p}|_S\le C_\gamma\cdot p^{\frac{1}{s+1}}\qquad (p\ge 1),
	\]
	because such $\gamma$ lies in the $(s+1)$-term of the lower central series 
	$\Gamma^{s+1}=[\Gamma,\Gamma^s]=[\Gamma,[\Gamma \dots]]$, and the growth rate on this subgroup
	is asymptotically scaled by $t^{s+1}$ (recall that in the asymptotic cone $G_\infty$ the homothety 
	$\delta_t$ acts by multiplication by $t^{j}$ on the $\mfg^{j}/\mfg^{j+1}$-subspace of $\mfg_\infty$).
	Therefore the length of the elements in (\ref{e:comm}) is bounded by
	\[
		C(m\cdot n^s)^{\frac{1}{s+1}}<C(\delta\cdot n^{s+1})^{\frac{1}{s+1}}=C\delta^{{\frac{1}{s+1}}}\cdot n
	\]
	which can be made $<2^{-k}\epsilon$ by choosing $\delta>0$ small enough.
\end{proof}

Finally, we are ready for the proof of the Ergodic Theorem along Polygonal Paths.

\medskip

\begin{proof}[Proof of Theorem~\ref{T:erg-poly}]\hfill{}\\
	Fix $\epsilon>0$ and let $\delta>0$ and $N$ be as in Lemma~\ref{L:pert}.
	
	Choose a small $\eta>0$ and let $M\in\bbN$ be large enough so that for each ${j}=1,\dots,k$ the set
	\[
		Y_{j}:=\left\{y\in X \mid \forall n\ge M:\quad |\frac{1}{n}\cdot c(T_{j}^n,y)-\phi(\ab{T}_{j})|<\epsilon\right\}
		\qquad\textrm{has}\qquad m(Y_{j})>1-\eta.
	\]
	Let $Z_k:=Y_k$, and apply Lemma~\ref{L:moving-target} with $E=Z_k$ to find $M_k\in\bbN$ so that the set
	\[
		(Z_k)_{M_k}^*:=\left\{z\in X \mid \forall n>M_k,\  \frac{\#\{n'<\delta\cdot n \mid  T^{n-n'}.z\in Z_k\}}{\delta\cdot n}>0\right\}
	\]
	satisfies $m((Z_k)_{M_k}^*)>m(Z_k)-\eta>1-2\eta$.
	Define $Z_{k-1}:=Y_{k-1}\cap (Z_k)_{M_k}^*$, and observe that
	\[
		m(Z_{k-1})>1-3\eta.
	\]
	One then continues inductively to define $Z_{{j}-1}:=(Z_{j})_{M_{{j}}}^*\cap Y_{{j}-1}$ (for ${j}=k-1,\dots,3,2$), 
	where $M_{{j}}\in\bbN$ is chosen large enough to ensure that the set
	\[
		(Z_{j})_{M_{{j}}}^*:=\left\{z\in X \mid \forall n>M_{j},\  
		\frac{\#\{n'<\delta\cdot n \mid  T^{n-n'}.z\in Z_{j}\}}{\delta\cdot n}>0\right\}
	\]
	has
	\[
		m((Z_{j})_{M_{{j}}}^*)>m(Z_{j})-\eta.
	\]
	The sets $Z_1,Z_2,\dots,Z_k$ that are defined in this manner satisfy 
	\[
		m(Z_1)>m(Z_2)-2\eta>m(Z_3)-4\eta>\dots>m(Z_k)-2(k-1)\eta>1-(2k-1)\eta.
	\]
	Let $N:=\max(M,M_1,\dots,M_k)$. Then for every $n>N$ and every $z\in Z_{j}$, there is $n_{j}<\delta\cdot n$
	so that $T^{n-n_{j}}.z\in Z_{{j}+1}$ and
	\[
		\left|\frac{1}{n}\cdot c(T_{j}^n,z)-\phi(\ab{T}_{j})\right|<\epsilon.
	\]
	Thus for $z$ in a set $Z_1$ of size $>1-(2k-1)\eta$ and every $n\ge N$, 
	there exist $n_1,\dots,n_k$ all bounded by $\delta\cdot n$, so that
	\[
		\left|\frac{1}{n}\cdot c(T_{j}^n,T_{{j}-1}^{n-n_{{j}-1}}\cdots T_2^{n-n_2}T_1^{n-n_1}.z)-\phi(\ab{T}_{j})\right|<\epsilon.
	\]
	Applying Lemma~\ref{L:pert} we have for each ${j}=1,\dots,k$:
	\[
		\begin{split}
			&d(T_{j}^n T_{{j}-1}^{n-n_{{j}-1}}\cdots T_2^{n-n_2}T_1^{n-n_1}, T_{j}^n T_{{j}-1}^n\cdots T_2^n T_1^n)<n\epsilon,\\
			&d(T_{{j}-1}^{n-n_{{j}-1}}\cdots T_2^{n-n_2}T_1^{n-n_1}, T_{{j}-1}^n\cdots T_2^n T_1^n)<n\epsilon.
		\end{split}
	\]
	So by Lemma~\ref{L:4points} and the Lipschitz property we have
	\begin{equation}\label{e:poly-pertpoly}
		\left|c(T_{j}^n,T_{{j}-1}^{n-n_{{j}-1}}\cdots T_2^{n-n_2}T_1^{n-n_1}.z)
		-c(T_{j}^n,T_{{j}-1}^n\cdots T_2^n T_1^n.z)\right|<2Kn\epsilon.
	\end{equation}
	Therefore for every $x\in Z_1$ and $n>N$ one has:
	\begin{equation}\label{e:eps}
		\left|\frac{1}{n}\cdot c(T_{j}^n,T_{{j}-1}^n\cdots T_2^n T_1^n.x)-\phi(\ab{T}_{j}) \right|<(2K+1)\epsilon
		\qquad ({j}=1,\dots,k).
	\end{equation}
	Applying this argument with a sequence of $\eta\to 0$, $m$-a.e. $x\in X$ would belong to at least one of the
	sets $Z_1$, and therefore would satisfy (\ref{e:eps}) for all $n>N(x,\epsilon)$.
	As $\epsilon>0$ was arbitrary, this proves that for $m$-a.e. $x\in X$ 
	\[
		\lim_{n\to\infty} \frac{1}{n}\cdot c(T_{j}^n,T_{{j}-1}^n\cdots T_2^n T_1^n.x)=\phi(\ab{T}_{j})\qquad ({j}=1,\dots,k)
	\]
	which in turn gives the convergence of the sum over $j=1,\dots,k$ to $\phi(\ab{T}_1)+\dots+\phi(\ab{T}_k)$.
	The $L^1$-convergence here follows by Lebesgue's Dominated convergence, because under the $L^\infty$-assumption 
	the terms are uniformly bounded.
	
	However, the latter conclusion of $L^1$-convergence does not require the assumption $c(\gamma,-)\in L^\infty(X,m)$,
	and holds under the weaker assumption $c(\gamma,-)\in L^1(X,m)$ for $\gamma\in\Gamma$.
	In the pointwise convergence argument, for every $x$ from a set $Z_1$ of large measure, for all $n$ large enough
	we compared the values of the cocycle along a polygonal path with that for a perturbed path (\ref{e:poly-pertpoly})
	and used Lemma~\ref{L:4points} to show that the values are close.
	In the $L^1$-context it is more natural to compare a polygonal path with the average of all perturbations:
	\[
		c(T_{j}^n,T_{{j}-1}^n\cdots T_1^n.x)- \frac{1}{(\delta n)^{{j}-1}}\cdot 
		\sum_{n_{{j}-1}=0}^{\lfloor \delta n\rfloor}\dots \sum_{n_1=0}^{\lfloor \delta n\rfloor} 
		c(T_{j}^n,T_{{j}-1}^{n-n_{{j}-1}}\cdots T_1^{n-n_1}.x)
	\]
	and replace Lemma~\ref{L:4points} by its $L^1$-version:
	\[
		\int_X |c(\tau,\,\alpha.x)-c(\tau',\,\beta.x)|\,dm(x)\le K_1\cdot (d(\alpha,\beta)+d(\tau\alpha,\tau'\beta))
	\] 
	where $K_1:=\max\{ \|c(s,-)\|_1 \mid s\in S\}$.
	We leave out the rather obvious details for this argument, as it it is not needed here. 
\end{proof}



\section{Proof of Theorems~\ref{T:NilKingman}, \ref{T:REM}} 
\label{sec:proof}

Throughout this section $\Gamma$, $\Gamma\acts (X,m)$, and $c:\Gamma\times X\to\bbR_+$ are as in Theorem~\ref{T:NilKingman},
and 
\[
	\phi:\ab{\Gamma}\otimes\bbR\cong\mfg_\infty\to \bbR_+,
	\qquad \Phi:G_\infty\to \bbR_+,
	\qquad d_\phi:G_\infty\times G_\infty\to \bbR_+
\] 
are as in Proposition~\ref{P:construction}.
We denote by $d_\infty$ the corresponding right-invariant, homogeneous metric on $G_\infty$ that 
appears in Pansu's Carnot-Carath\'eodory construction.
We denote 
\[
	\Ball(g,\epsilon):=\{g'\in G_\infty \mid d_\infty(g,g')<\epsilon\}
\]
the corresponding balls in $G_\infty$.
Consider the functions
\[
	\begin{split}
		&c^*(g,x):=\lim_{\epsilon\searrow 0}\ \limsup_{t\to\infty}\ \sup_{\scl{t}{\gamma}\in \Ball(g,\epsilon)}\frac{1}{t}c(\gamma,x),\\
		&c_*(g,x):=\lim_{\epsilon\searrow 0}\ \liminf_{t\to\infty}\ \inf_{\scl{t}{\gamma}\in \Ball(g,\epsilon)}\frac{1}{t}c(\gamma,x).
	\end{split}
\]
While this is not necessary for our argument, it is impossible to ignore the fact that $c^*(g,-)$ and $c_*(g,-)$ are a.e. constant.
\begin{lemma}\label{L:c-star}
	For each $g\in G_\infty$ the functions $c^*(g,-)$, $c_*(g,-)$ are $m$-a.e. constants, denoted $c^*(g)$, $c_*(g)$, respectively.
\end{lemma}
\begin{proof}
	For any fixed $g\in G_\infty$ the functions $c_*(g,-), c^*(g,-):X\to\bbR_+$ are measurable.
	Fix $\gamma_0\in \Gamma$. Then for any $\epsilon>0$ for all $t>t(g,\gamma_0,\epsilon)$ one has
	\[
		\ssc{t}{\gamma}\in \Ball(g,\epsilon)\qquad\Longrightarrow\qquad 
		\ssc{t}{\gamma\gamma_0},\ \ssc{t}{\gamma\gamma_0^{-1}}\in \Ball(g,2\epsilon).
	\] 
	Since for every $x\in X$
	\[
		\frac{1}{t}c(\gamma\gamma_0,x)\le \frac{1}{t}c(\gamma,\gamma_0.x)+\frac{1}{t}c(\gamma_0,x)
	\]
	it follows that $c^*(g,x)\le c^*(g,\gamma_0.x)$ and $c_*(g,x)\le c_*(g,\gamma_0.x)$.
	Applying the same argument to $\gamma_0^{-1}$ and $\gamma_0.x$ we observe that $c_*(g,-)$ and $c^*(g,-)$
	are measurable $\Gamma$-invariant functions. 
	Hence they are a.e. constants, because $\Gamma\acts (X,m)$ is ergodic.
\end{proof}

In the following subsections we shall proceed in the following steps:
\begin{enumerate}
	\item Show that $c^*(g)\le \Phi(g)$ for all $g\in G_\infty$.
	\item Show that $\Phi(g)\le c_*(g)$ for all $g\in G_\infty$.
	\item The obvious inequality $c_*\le c^*$ combined with the above implies that $t_i^{-1}\cdot c(\gamma_i,x)\to \Phi(g)$
	whenever $\scl{t_i}{\gamma_i}\to g$ in $G_\infty$.
	We shall show that for a.e. $x\in X$ the above convergence is uniform over $g\in \Ball(e,1)$ and will deduce
	Theorem~~\ref{T:NilKingman} by rescaling.
	\item We will prove Theorem~\ref{T:REM} by combining the ideas of the previous steps.
\end{enumerate}

Let $X_0\subset X$ be the set of $x\in X$ for which $c^*(g,x)=c^*(g)$, $c_*(g,x)=c_*(g)$, and
Theorem~\ref{T:erg-poly} holds for all $k\in\bbN$ and every choice of $T_1,\dots,T_k\in\Gamma$.
We imposed countably many condition where each holds $m$-a.e., therefore $m(X\setminus X_0)=0$.

\medskip

\subsection{The upper bound: $c^*(g)\le \Phi(g)$} 
\label{sub:upperbound}\hfill{}\\
Fix $x\in X_0$, and assume, towards contradiction, that 
there exists $\eta>0$ and sequences $t_i\to\infty$ and $\gamma_i\in \Gamma$ so that
\begin{equation}\label{e:overshoot}
	\lim_{i\to\infty} \ssc{t_i}{\gamma_i}=g,\qquad
	\textrm{while}
	\qquad 
	\frac{1}{t_i}c(\gamma_i,x)>d_\phi(e,g)+\eta.
\end{equation}
Fix a small $\epsilon>0$, namely $\epsilon=\eta/(K+3)$, where $K$ is as in Theorem~\ref{T:NilKingman}(i).
Choose a $\phi$-geodesic $\xi:[0,1]\to G_\infty$, i.e. a smooth admissible curve such that
\[
	\xi(0)=e,\qquad \xi(1)=g,\qquad \len_\phi(\xi)=\Phi(g)
\]
(we could choose any smooth curve from $e$ to $g$  with $\len_\phi(\xi)<\Phi(g)+\epsilon$ with a sufficiently small $\epsilon>0$).
Applying Proposition~\ref{P:curve-to-polygon} we find $k\in\bbN$, elements $T_1,\dots,T_k\in\Gamma$,
and a multiple $p\in\bbN$ that give $\epsilon$-approximation to the curve $\xi$.
Set
\[
	n_i:=\lfloor \frac{t_i}{p} \rfloor,\qquad S_i:=T_k^{n_i}\cdots T_2^{n_i} T_1^{n_i}.
\]
Note that 
\[
	\limsup_{i\to\infty}\frac{1}{t_i}\cdot d(S_i,\gamma_i)=\limsup_{i\to\infty}\ d_{\infty}(\ssc{t_i}{S_i},\ssc{t_i}{\gamma_i})
	=\limsup_{i\to\infty}\ d_{\infty}(\ssc{t_i}{S_i},g)<\epsilon.
\]
Since $c(-,x):\Gamma\to \bbR_+$ is $K$-Lipschitz, we have for all sufficiently large $i\gg 1$.
\[
	\begin{split}
		\frac{1}{n_ip}&\cdot \sum_{j=1}^k c(T_j^{n_i},\, T^{n_i}_{j-1}\cdots T_1^{n_i}.x)
		\ge\frac{1}{n_ip}\cdot c(S_i,x)>\frac{1}{t_i}\cdot c(S_i,x)-\epsilon\\
		&>\frac{1}{t_i}\cdot c(\gamma_i,x)-K\cdot\epsilon-\epsilon>\Phi(g)+(\eta-(K+1)\epsilon).
	\end{split}
\]
The above inequalities use sub-additivity, the fact that $n_ip/t_i\to 1$, the Lipschitz property of $c(-,x)$, and the
assumption (\ref{e:overshoot}) that we try to refute.
Applying Theorem~\ref{T:erg-poly} we have 
\[
	\lim_{i\to\infty}\ \frac{1}{n_ip}\cdot\sum_{{j}=1}^k c(T_{j}^{n_i},T^{n_i}_{{j}-1}\cdots T_1^{n_i}.x)
	=\frac{1}{p}\cdot\left(\phi(\ab{T}_k)+\dots+\phi(\ab{T}_1)\right).
\]
However, by part (ii) of Proposition~\ref{P:curve-to-polygon}, one also has
\[
	\frac{1}{p}\cdot\left(\phi(\ab{T}_k)+\dots+\phi(\ab{T}_1)\right)<\len_\phi(\xi)+\epsilon<\Phi(g)+\epsilon.
\] 
This leads to a contradiction, due to our choice of $\epsilon=\eta/(K+3)$.
Thus (\ref{e:overshoot}) is impossible.


\medskip

\subsection{The lower bound: $c_*(g)\ge \Phi(g)$} 
\label{sub:lowerbound}\hfill{}\\
Let us now prove the inequality $c_*(g)\ge \Phi(g)$. 
Fix $g\in G_\infty$, $x\in X_0$, and assume, towards contradiction, that 
there exists $\eta>0$ and sequences $t_i\to\infty$ and $\gamma_i\in \Gamma$ so that
\begin{equation}\label{e:undershoot}
	\lim_{i\to\infty} \ssc{t_i}{\gamma_i}=g\qquad
	\textrm{while}
	\qquad 
	\frac{1}{t_i}c(\gamma_i,x)<\Phi(g)-\eta.
\end{equation}
We take a small $\epsilon>0$ and an associated finite set $F\subset\Gamma$ as in condition (ii) of Theorem~\ref{T:NilKingman}.
Apply the following argument to each $\gamma_i$ from the sequence satisfying (\ref{e:undershoot}).

Each $\gamma_i$ can be written as a product 
\[
		\gamma_i=\delta_{i,s_i}\cdots \delta_{i,2}\delta_{i,1},
\]
where $\delta_{i,j}\in F$ for all $1\le j\le s_i$ and 
\[
	\sum_{j=1}^{s_i} c(\delta_{i,j},\,\delta_{i,j-1}\cdots \delta_{i,1}.x)<(1+\epsilon)\cdot c(\gamma_i,x).
\]
Consider the sequence of points
\[
	g_{i,j}:=\ssc{t_i}{\delta_{i,j}\cdots \delta_{i,1}}\qquad (j=1,\dots,s_i).
\]
Define a piecewise $d_\infty$-geodesic curve  
\[
	\xi_i:[0,1]\overto{} G_\infty 
\]
connecting $e$ to $\ssc{t_i}{\gamma_i}=g_{i,s_i}$ via the points $g_{i,j}$, which are to be visited at times
\[
	\xi_i(\frac{c_{i,1}+\dots+c_{i,j}}{c_{i,1}+\dots+c_{i,s_i}})=g_{i,j},
	\qquad\textrm{where}\qquad c_{i,j}:=c(\delta_{i,j},\,\delta_{i,r-1}\cdots \delta_{i,1}.x).
\]
Between these times $\xi_i(-)$ follows an appropriately rescaled $g_\infty$-geodesic.
So $\xi_i$ traces in $G_\infty$ the points associated to partial products representing a discrete path from $e$ to $\gamma_i$,
with time parameter chosen according to the $c_{i,j}$-steps. 

The bi-Lipschitz condition for $c(-,x)$ in terms of $d$ (condition~(i) in Theorem~\ref{T:NilKingman}), 
implies that $\xi_{i}:[0,1]\to G_\infty$ is a uniformly Lipschitz sequence of maps with $\xi_i(0)=e$.
Hence by Arzela-Ascoli, upon passing to a subsequence, we may assume that $\xi_i$ converge (uniformly) to a  
Lipschitz curve
\[
	\xi:[0,1]\ \overto{}\ G_\infty, \qquad \xi(0)=e,\qquad \xi(1)=g.
\] 
Since we are working towards a contradiction to (\ref{e:undershoot}) which holds for sub-sequences, we may assume
that $\xi_i\to \xi$ without complicating our notations any further.
 
With the Lipschitz curve $\xi$ at hand and small $\epsilon>0$, Proposition~\ref{P:curve-to-polygon} provides
$k\in\bbN$, $T_1,\dots, T_k\in\Gamma$ and $p\in\bbN$ that give $\epsilon$-good approximation for the curve $\xi$:
In particular, for large $i\gg 1$ and $n_i:=\lfloor t_i/p \rfloor$ one has
\[
	\sum_{j=1}^k\ d_\infty\left(\ssc{n_ip}{(T^{n_i}_j\cdots T^{n_i}_2T^{n_i}_1}),\,\xi(\frac{j}{k})\right)<\epsilon.
\]
For each $i\in\bbN$, choose $0=r_i(0)<r_i(1)<\dots<r_i(k)=s_i$ so that for $j=1,\dots,k$:
\[
	\frac{c_{i,r_i(j-1)+1}+\dots+c_{i,r_i(j)}}{c_{i,1}+\dots+c_{i,s_i}}\ \overto{}\ \frac{1}{k},
\]
and write $\ \gamma_i=\pi_{i,k}\cdots \pi_{i,2}\pi_{i,1}\ $ with $\ \pi_{i,j}:=\delta_{i,r_i(j)}\cdots \delta_{i,r_i(j-1)+1}$.
Then
\[
	\pi_{i,j}\cdots \pi_{i,2}\pi_{i,1}=\delta_{i,r_i(j)}\cdots \delta_{i,2}\delta_{i,1}\qquad (j=1,\dots,k).
\]
We have
\[
	\xi(\frac{j}{k})=\lim_{i\to\infty}\ \xi_i(\frac{j}{k})
	=\lim_{i\to\infty}\ \ssc{n_ip}{\pi_{i,j}\cdots \pi_{i,2}\pi_{i,1}}. 
\]
Thus for all large enough $i$:
\[
	\sum_{j=1}^k\ d_\infty\left(\ssc{n_ip}{T^{n_i}_j\cdots T^{n_i}_2T^{n_i}_1},\ssc{n_ip}{\pi_{i,j}\cdots \pi_{i,2}\pi_{i,1}}\right)<\epsilon
\]
and therefore for all large enough $i$:
\[
	\sum_{j=1}^k\ \frac{1}{n_ip}\cdot d\left(T^{n_i}_j\cdots T^{n_i}_2T^{n_i}_1,\pi_{i,j}\cdots \pi_{i,2}\pi_{i,1}\right)<\epsilon.
\]
We now apply Lemma~\ref{L:4points} with 
\[
	\alpha=T_{j-1}^{n_i}\cdots T^{n_i}_1,\qquad \tau=T_j^{n_i},\qquad
	\beta=\pi_{i,j-1}\cdots\pi_{i,1},\qquad\tau'=\pi_{i,j}
\]
to deduce that for all large enough $i$:
\[
	\sum_{j=1}^k\ \frac{1}{n_ip}\cdot\left|c(T^{n_i}_j,\,\,T_{j-1}^{n_i}\cdots T^{n_i}_1.x)  
	- c(\pi_{i,j},\,\pi_{i,j-1}\cdots\pi_{i,1}.x)\right|<2K\epsilon.
\]
For $i\gg1$ we have
\[
	\begin{split}
		\frac{1}{t_i}&\cdot c(\gamma_i,x)>\frac{n_ip}{(1+\epsilon) t_i}\cdot 
			\frac{1}{n_ip}\cdot \sum_{r=1}^{s_i} c(\delta_{i,r},\delta_{i,r-1}\cdots\delta_{i,1}.x)\\
		&\ge (1-\epsilon)\cdot \frac{1}{n_ip} \cdot
			\sum_{j=1}^k \sum_{r=r_i(j-1)+1}^{r_i(j)} c(\delta_{i,r},\delta_{i,r-1}\cdots\delta_{i,1}.x)\\
		&\ge (1-\epsilon)\cdot \frac{1}{n_ip} \cdot\sum_{j=1}^k c(\pi_{i,j},\,\pi_{i,j-1}\cdots\pi_{i,1}.x)\\
		&>(1-\epsilon)\cdot \left(\frac{1}{n_ip} \cdot
			\sum_{j=1}^k c(T_j^{n_i},\,T_{j-1}^{n_i}\cdots T^{n_i}_1.x)-2K\epsilon\right)
	\end{split}
\]
using sub-additivity of $c$ in the third inequality.
Theorem~\ref{T:erg-poly} gives
\[
	\lim_{i\to\infty}\ 
	\frac{1}{n_ip}\cdot \sum_{j=1}^k c(T_j^{n_i},T_{j-1}^{n_i}\cdots T^{n_i}_1.x)
	=\frac{1}{p}\left(\phi(\ab{T}_k)+\dots+\phi(\ab{T}_1)\right)
	>\len_\phi(\xi)-\epsilon.
\] 
Since $\xi$ is only one of many possible admissible curves connecting $\xi(0)=e$ to $\xi(1)=g$ 
(and most likely is sub-optimal in terms of the $\phi$-length), one has
\[
	\len_\phi(\xi)\ge d_\phi(e,g)=\Phi(g).
\]
Therefore we deduce
\[
	\liminf_{i\to\infty}\ \frac{1}{t_i}\cdot c(\gamma_i,x)\ge (1-\epsilon)\cdot \left(\Phi(g)-(2K+1)\epsilon\right).
\]
A choice of small enough $\epsilon>0$ contradicts (\ref{e:undershoot}).
This proves the claimed inequality 
\[
	\Phi(g)\le c_*(g).
\]

\medskip

\subsection{Proof of Theorem~\ref{T:NilKingman}} 
\label{sub:proof_of_A}\hfill{}\\
The results of the two previous subsections giving $c^*(g)\le \Phi(g)\le c_*(g)$, combined with the trivial 
inequality $c_*(g)\le c^*(g)$, show 
\[
	c_*(g)=c^*(g)=\Phi(g).
\]
Equivalently
\begin{equation}\label{e:convergence}
	\lim_{\epsilon\searrow 0}\ \limsup_{t\to\infty}\ 
	\sup\left\{|\frac{1}{t}\cdot c(\gamma,x)-\Phi(g)|\ :\ \ssc{t}{\gamma}\in \Ball(g,\epsilon)\right\}=0.
\end{equation}
We need to prove that for $m$-a.e. $x\in X$ (or rather every $x\in X_0$) one has
\[
	\forall \epsilon>0,\quad \exists R<\infty:\qquad
	|\gamma|_S\ge R\qquad\Longrightarrow\qquad
	\left|c(\gamma,x)-\Phi(\scl{1}{\gamma})\right|<\epsilon\cdot |\gamma|_S.
\]
Indeed, if the claim were not true, we could find $\epsilon_0>0$ and a sequence $\gamma_n\in\Gamma$
with $|\gamma_n|_S\to \infty$, so that 
\[
	\left|c(\gamma_n,x)-\Phi(\scl{1}{\gamma_n})\right|\ge \epsilon_0\cdot |\gamma_n|_S.
\]
The sequence 
\[
	g_n:=\ssc{|\gamma_n|_S}{\gamma_n}
\]
has
\[ 
	\limsup_{n\to\infty} d_\infty(g_n,e)\le 1.
\]
Hence $\{g_n \mid n\in\bbN\}$ is bounded.
Since balls in $G_\infty$ are precompact, there is a subsequence 
$\gamma_{n_i}$ converging to some $g\in G_\infty$ (in fact $g\in \Ball(e,1)$).
Denote $t_i:=|\gamma_{n_i}|_S$. We note that 
\[
	d_\infty(\ssc{t_i}{\gamma_{n_i}},\delta_{\frac{1}{t_i}}(\scl{1}{\gamma_{n_i}}))\to 0
\]
and therefore
\[
	\lim_{i\to\infty}
	\frac{1}{t_i}\cdot \Phi(\scl{1}{\gamma_{n_i}})=\lim_{i\to\infty}\Phi(\delta_{\frac{1}{t_i}}(\scl{1}{\gamma_{n_i}}))
	= \Phi(\lim_{i\to\infty}\ssc{t_i}{\gamma_{n_i}})=\Phi(g).
\]
Finally (\ref{e:convergence}) implies
\[
	\frac{1}{t_i}\left| c(\gamma_{n_i},x)-t_i\cdot \Phi(g)\right|\to 0
\]
contrary to the assumption. This proves Theorem~\ref{T:NilKingman}.


\medskip

\subsection{Proof of Theorem~\ref{T:REM}} 
\label{sub:proof_of_theorem_t:rem}\hfill{}\\
The main claim is that given any $g,g'\in G_\infty$ and sequences $t_i\to\infty$, $\gamma_i,\gamma'_i\in\Gamma$,
so that
\begin{equation}\label{e:ggprime}
	\lim_{i\to\infty}\ \ssc{t_i}{\gamma_i}\ \overto{}\ g,\qquad \lim_{i\to\infty}\ \ssc{t_i}{\gamma'_i}=g'
\end{equation}
one necessarily has for every $x\in X_0$:
\[
	\lim_{i\to\infty}\ \frac{1}{t_i}\cdot c(\gamma'_i\gamma_i^{-1},\,\gamma_i.x)\ =\ d_\phi(g,g').
\]
To show this we employ a variant on the upper bound argument \S\ref{sub:upperbound} and on
the lower bound argument \S\ref{sub:lowerbound}. 
In both of these arguments we use a fixed admissible curve $\xi_0$ connecting $e$ to $g$ in $G_\infty$,
concatenated with an appropriate curve $\xi$ connecting $g$ to $g'$ in $G_\infty$.

Denote by $\xi_1:[0,2]\to G_\infty$ the curve that connects $e$ to $g'$ via $g$:
\[
	\xi_1(0)=e,\qquad \xi_1(1)=g,\qquad \xi_1(2)=g';\qquad \xi_1(s+1)=\xi(s).
\]
Fix a  small $\epsilon>0$, and apply Proposition~\ref{P:curve-to-polygon} to $\xi_1$ to find
\[
	T_1,\dots,T_{2k}\in\Gamma,\qquad p\in\bbN,
\] 
so that 
\begin{equation}\label{e:lengthgg}
	\left|\frac{1}{p}\cdot\sum_{j=k+1}^{2k} \phi(\ab{T}_j)-\len_\phi(\xi)\right|<\epsilon,
\end{equation}
while for all $n\ge n_0$
\begin{equation}\label{e:followgg}
		\sum_{j=1}^{k}\ d_\infty\left(\ssc{np}{T_{k+j}^n\cdots T_2^n T_1^n},\,\xi(\frac{j}{k})\right)<\epsilon.
\end{equation}
Note that the last condition is a trivial consequence of the estimate on
\[
	\sum_{j=1}^{k}\ d_\infty\left(\ssc{np}{T_j^n\cdots T_2^n T_1^n},\,\xi_1(\frac{j}{k})\right)
	+\sum_{j=1}^{k}\ d_\infty\left(\ssc{np}{T_{k+j}^n\cdots T_2^n T_1^n},\,\xi_1(\frac{j}{k}+1)\right),
\]
while (\ref{e:lengthgg}) can be obtained from approximating the $\phi$-lengths of $\xi_1$ and $\xi_0$ by
\[
	\frac{1}{p}\cdot\sum_{j=1}^{2k} \phi(\ab{T}_j),\qquad\textrm{and}\qquad \frac{1}{p}\cdot\sum_{j=1}^{k} \phi(\ab{T}_j)
\]
and the obvious relation 
\[
	\len_\phi(\xi)=\len_\phi(\xi_1)-\len_\phi(\xi_0).
\]
Next, choosing $\xi$ to be a $\phi$-geodesic connecting $g$ to $g'$, and taking $n_i=\lfloor t_i/p\rfloor$, we get
\[
	\lim_{i\to\infty}\ d_\infty(\ssc{n_ip}{\gamma_i},g)=\lim_{i\to\infty}\ d_\infty(\ssc{n_ip}{\gamma'_i},g')=0
\] 
and
\[
	\limsup_{i\to\infty}\frac{1}{n_ip}\cdot d(\gamma_i,T^{n_i}_{k}\cdots T^{n_i}_1)\le \epsilon,
	\qquad
	\limsup_{i\to\infty}\frac{1}{n_ip}\cdot d(\gamma'_i,T^{n_i}_{2k}\cdots T^{n_i}_1)\le \epsilon.
\]
Following the same argument as 
in \S\ref{sub:upperbound} (using Lemma~\ref{L:4points} and Theorem~\ref{T:erg-poly}), we have
\[
	\begin{split}
		\limsup_{i\to\infty}\ &\frac{1}{t_i}\cdot c(\gamma'_i\gamma_i^{-1},\gamma_i.x)
			\le \limsup_{i\to\infty}\ \frac{1}{n_ip}\cdot c(T_{2k}^{n_i}\cdots T^{n_i}_{k+1},\,T_{k}^{n_i}\cdots T_1^{n_i}.x)+2K\epsilon\\
		&\le \lim_{i\to\infty}\ \frac{1}{n_ip}\cdot\sum_{j=1}^k 
			c(T_{k+j}^{n_i},\,T^{n_i}_{k+j-1}\cdots T_1^{n_i}.x)+2K\epsilon=\frac{1}{p}\cdot\sum_{j=1}^k \phi(\ab{T}_{k+j})+2K\epsilon\\
		&<\len_\phi(\xi)+(2K+1)\epsilon=d_\phi(g,g')+(2K+1)\cdot\epsilon.
	\end{split}
\]
Since $\epsilon>0$ was arbitrary, this shows the upper bound:
\[
	\limsup_{i\to\infty}\ \frac{1}{t_i}\cdot c(\gamma'_i\gamma_i^{-1},\gamma_i.x)\le d_\phi(g,g').
\]
The lower bound, 
\[
	\liminf_{i\to\infty}\ \frac{1}{t_i}\cdot c(\gamma'_i\gamma_i^{-1},\gamma_i.x)\ge d_\phi(g,g')
\]
is trivial if $g=g'$. Hence we assume $g\ne g'$ which implies that $|\gamma'_i\gamma_i^{-1}|_S\to\infty$.
We now use the innerness assumption (corresponding to condition (ii) in Theorem~\ref{T:NilKingman}).
Fix an arbitrary small $\epsilon>0$ and rewrite $\gamma'_i\gamma_i^{-1}$ as a product of 
\[
	\gamma'_i\gamma_i^{-1}=\delta_{i,s_i}\cdots \delta_{i,1},
	\qquad \textrm{while}\qquad 
	\sum_{r=1}^{s_i} c(\delta_{i,r},\,\delta_{i,r-1}\cdots\delta_{i,1}\gamma_i.x)
		<(1+\epsilon)\cdot c(\gamma'_i\gamma_i^{-1},\gamma_i.x).
\]
where $\delta_{i,j}$ belong to a fixed finite set $F\subset \Gamma$ (depending on $\epsilon$ and $x\in X_0$).
One then proceeds as in \S\ref{sub:lowerbound} to construct a uniformly Lipschitz sequence of piecewise geodesic curves
connecting $g$ to $\approx g'$, and to use Arzela-Ascoli to pass to a convergent subsequence
that produces a Lipschitz curve 
\[
	\xi:[0,1]\to G_\infty,\qquad \xi(0)=g,\qquad \xi(1)=g'.
\]
We are going to concatenate $\xi_0$ with $\xi$ to get $\xi_1:[0,2]\to G_\infty$ as before.
The long products $\gamma'_i\gamma_i^{-1}=\delta_{i,s_i}\cdots \delta_{i,1}$ can be sub-partitioned so that
\[
	\gamma'_i\gamma_i^{-1}=\pi_{i,k}\cdots \pi_{i,2}\pi_{i,1}
\]
while for $j=1,\dots,k$ one has
\[
	\xi(\frac{j}{k})=\xi_1(\frac{k+j}{k})=\lim_{i\to\infty}\ \ssc{t_i}(\pi_{i,j}\cdots \pi_{i,1}\gamma_i).
\]
We now invoke the $T_1,\dots,T_{2k}$ and $p\in\bbN$ satisfying (\ref{e:lengthgg}) and (\ref{e:followgg}).
One has
\[
	\limsup_{i\to\infty}\ \sum_{j=1}^k \frac{1}{n_ip}\cdot d(T_{k+j}^{n_i}\cdots T_1^{n_i}\gamma_i,\,\pi_{i,j}\cdots \pi_{i,1}\gamma_i)\le \epsilon
\]
The sub-additivity gives
\[
	\sum_{j=1}^{k} c(\pi_{i,j},\pi_{i,j-1}\cdots\pi_{i,2}\pi_{i,1}\gamma_i.x)
	\le \sum_{r=1}^{s_i} c(\delta_{i,r},\,\delta_{i,r-1}\cdots\delta_{i,1}\gamma_i.x).
\]
Combining these facts, one shows that for a subsequence of the given $t_i,\gamma_i,\gamma'_i$ one has:
\[
	\begin{split}
		\liminf_{i\to\infty}\ &\frac{1}{t_i}\cdot c(\gamma'_i\gamma_i^{-1},\,\gamma_i.x)
		\ge \liminf_{i\to\infty} \frac{1}{1+\epsilon} 
		\cdot\frac{1}{n_ip}\sum_{j=1}^{k} c(\pi _{i,j},\pi_{i,j-1}\cdots\pi_{i,2}\pi_{i,1},\,\gamma_i.x)\\
		&\ge (1-\epsilon) \cdot\left(\liminf_{i\to\infty}\ \frac{1}{n_ip}\cdot
			\sum_{j=1}^{k} c(T_{k+j}^{n_i},\, T^{n_i}_{k+j-1}\cdots T_1^{n_i}.x)-2K\epsilon\right)\\
		&= (1-\epsilon)\cdot\left(\frac{1}{p}\cdot\left(\phi(\ab{T}_{2k})+\dots+\phi(\ab{T}_{k+1})\right)-2K\epsilon\right)\\
		&\ge (1-\epsilon)\cdot (\len_\phi(\xi)-(2K+1)\epsilon)
		\ge (1-\epsilon)\cdot (d_\phi(g,g')-(2K+1)\epsilon).
	\end{split}
\]
Since $\epsilon>0$ is arbitrary, and any subsequence of $t_i,\gamma_i,\gamma_i'$ contains a sub-sub-sequence
satisfying the above, it follows 
\[
	\liminf_{i\to\infty}\ \frac{1}{t_i}\cdot c(\gamma'_i\gamma_i^{-1},\gamma_i.x)\ge d_\phi(g,g').
\]
In view of the $\limsup$ inequality, the lower bound is also proven.
As in the proof of Theorem~\ref{T:NilKingman} one can easily deduce that for $m$-a.e. $x\in X$ for every $\epsilon>0$
there is $R<\infty$ so that for $|\gamma|_S,|\gamma'|_S>R$ one has 
\[
	|c(\gamma'\gamma^{-1},\gamma.x)-d_\phi(\scl{1}{\gamma},\scl{1}{\gamma'})|<\epsilon\cdot \max(|\gamma|_S,|\gamma'|_S).
\]
This completes the proof of Theorem~\ref{T:REM}.



\end{document}